\documentclass[a4paper,10pt]{amsart}
\usepackage[utf8x]{inputenc}
\usepackage{amssymb,amsmath, amsthm, amsfonts,epsfig,enumerate, wrapfig, tikz, color, tensor}
\usetikzlibrary{decorations.markings}
\usetikzlibrary{patterns}
\newcounter{res}[section]
\numberwithin{res}{section}
\newtheorem{thm}[res]{Theorem}
\newtheorem{lem}[res]{Lemma}
\newtheorem{prop}[res]{Proposition}
\newtheorem*{klem}{Key lemma}

\theoremstyle{definition}
\newtheorem{notation}[res]{Notation}
\newtheorem{dfn}[res]{Definition}
\newtheorem{req}[res]{Remark}

\newcommand{\NN}{\mathbb N} 
\newcommand{\ZZ}{\mathbb Z} 
 
\newcommand{\QQ}{\mathbb Q}
\newcommand{\RR}{\mathbb R} 
\newcommand{\F}{\mathcal{F}}

\def \m {^{-1}}

\newcommand{\cjg}[1]{\overline{#1}}

\renewcommand {\epsilon}{\varepsilon}

\renewcommand {\leq}{\leqslant}

\renewcommand {\bar}{\cjg}

\renewcommand\marginpar[1]{}%\oldmarginpar{\color{red}\raggedleft\footnotesize #1}}
\newcommand\eqdef{\ensuremath{\stackrel{\textrm{def}}{=}}}
\newcommand\kup[1]{\ensuremath\left\langle #1 \right\rangle}

\newcommand{\imagesfolder}{.}
\newcommand{\ie}{i.~e.~}

\newcommand{\foamsquarecircle}[1][0.6]{\vcenter{\hbox{\!
 \begin{tikzpicture}[scale={#1},decoration={markings, mark=at
     position 0.5 with {\arrow{>}}},postaction={decorate}]
      \draw[dotted] (1.5,1.5) circle (1.414 cm);
      \draw[postaction = {decorate}] (0.5,0.5)--(1,1);
      \draw[postaction = {decorate}] (2.5,2.5)--(2,2);
      \draw[postaction = {decorate}] (2,1)--(1,1);
      \draw[postaction = {decorate}] (2,1)--(2,2);
      \draw[postaction = {decorate}] (2,1)--(2.5,0.5);
      \draw[postaction = {decorate}] (1,2)--(2,2);
      \draw[postaction = {decorate}] (1,2)--(1,1);
      \draw[postaction = {decorate}] (1,2)--(0.5,2.5);
 \end{tikzpicture}}}}
\newcommand{\websquare}[1][0.6]{\vcenter{\hbox{\!
 \begin{tikzpicture}[scale={#1},decoration={markings, mark=at
     position 0.5 with {\arrow{>}}},postaction={decorate}]
      \draw[postaction = {decorate}] (0.5,0.5)--(1,1);
      \draw[postaction = {decorate}] (2.5,2.5)--(2,2);
      \draw[postaction = {decorate}] (2,1)--(1,1);
      \draw[postaction = {decorate}] (2,1)--(2,2);
      \draw[postaction = {decorate}] (2,1)--(2.5,0.5);
      \draw[postaction = {decorate}] (1,2)--(2,2);
      \draw[postaction = {decorate}] (1,2)--(1,1);
      \draw[postaction = {decorate}] (1,2)--(0.5,2.5);
 \end{tikzpicture}}}}
\newcommand{\webtwovert}[1][0.6]{
\vcenter{\hbox{\!
 \begin{tikzpicture}[scale={#1},decoration={markings, mark=at
     position 0.5 with {\arrow{>}}},postaction={decorate}]
      \draw[postaction= {decorate}] (0.5,0.5) .. controls (1,1) and (1,2) .. (0.5, 2.5);
      \draw[postaction= {decorate}] (2.5,2.5) .. controls (2,2) and (2,1) .. (2.5, 0.5);
 \end{tikzpicture}}}}

\newcommand{\webtwohori}[1][0.6]{\vcenter{\hbox{\!
 \begin{tikzpicture}[scale={#1},decoration={markings, mark=at
     position 0.5 with {\arrow{>}}},postaction={decorate}]
      \draw[postaction= {decorate}] (0.5,0.5) .. controls (1,1) and (2,1) .. (2.5, 0.5);
      \draw[postaction= {decorate}] (2.5,2.5) .. controls (2,2) and (1,2) .. (0.5, 2.5);
 \end{tikzpicture}}}}
\newcommand{\webvert}[1][0.6]{\vcenter{\hbox{\!
 \begin{tikzpicture}[scale={#1},decoration={markings, mark=at
     position 0.5 with {\arrow{>}}},postaction={decorate}]
      \draw[postaction= {decorate}] (0.5,0.5) -- (0.5, 2.5);
  \end{tikzpicture}}}}

\newcommand{\webbigon}[1][0.6]{\vcenter{\hbox{\!
 \begin{tikzpicture}[scale={#1},decoration={markings, mark=at
     position 0.5 with {\arrow{>}}},postaction={decorate}]
      \draw[postaction= {decorate}] (0.5,0.5) -- (0.5, 1);
      \draw[postaction= {decorate}] (0.5,2) -- (0.5, 2.5);
      \draw[postaction= {decorate}] (0.5,2) ..controls (0,1.5) and (0,1.5).. (0.5, 1);
      \draw[postaction= {decorate}] (0.5,2) ..controls (1,1.5) and (1,1.5).. (0.5, 1);
  \end{tikzpicture}}}}
\newcommand{\webcircle}[1][0.6]{\vcenter{\hbox{\!
 \begin{tikzpicture}[scale={#1},decoration={markings, mark=at
     position 0.5 with {\arrow{>}}},postaction={decorate}]
      \draw[postaction= {decorate}] (1.5,1.5) circle (1cm);
 \end{tikzpicture}}}}
\newcommand{\webcirclereverse}[1][0.6]{\vcenter{\hbox{\!
 \begin{tikzpicture}[scale={#1},decoration={markings, mark=at
     position 0.5 with {\arrow{<}}},postaction={decorate}]
      \draw[postaction= {decorate}] (1.5,1.5) circle (1cm);
 \end{tikzpicture}}}}

\begin{document}
\author{Louis-Hadrien Robert}
\title[Indecomposable modules for Khovanov-Kuperberg algebras]{A large family of indecomposable projective modules for the 
Khovanov-Kuperberg algebras of $sl_3$-webs}
\address{Louis-Hadrien Robert \\ IMJ \\ 175 rue du Chevaleret \\ 75013 Paris  \\ France}
\email{lrobert@math.jussieu.fr}
\date{\today}
\maketitle
%\end{document}
%\begin{document}
\begin{abstract}
  In this paper we recall the construction from Mackaay -- Pan -- Tubbenhauer of the algebras $K^{\epsilon}$ which allow to understand the $sl_3$ homology for links in a local way (\ie for tangles). Then, by studying the combinatorics of the Kuperberg bracket, we give a large family of non-elliptic webs whose associated projective modules are indecomposable over these algebras.  
\end{abstract}
%\pagestyle{fancy}
%\pagestyle{fancyplain}
%\fancyhf{}
%\lhead{\fancyplain{}{\small \theauthor}}
%\rhead{\fancyplain{}{\small \today}}
%\cfoot{\fancyplain{}{\small \thepage}}
%\maketitle
%\tableofcontents
\section{Introduction}
\label{sec:introduction}
In 1992, Greg Kuperberg (\cite{MR1403861}), following Reshetikhin-Turaev's point of view gave a diagrammatic presentation of a basis of invariant spaces of tensor products of the canonical representation (and its dual) of the quantum group $U_q(sl_3)$. This basis is given by non-elliptic webs \ie  oriented trivalent plane graphs with certain geometric conditions. Along the way he showed that the  webs thought of in terms of $U_q(sl_3)$ representation theory satisfy some relations and this leads to the definition of a Laurent polynomial associated with closed webs, the Kuperberg bracket. In 1999, Khovanov (\cite{MR2100691}) gave a categorification of this Laurent polynomial in a TQFT manner~: with a web he associates a graded $\ZZ$-module with graded dimension given by the Kuperberg bracket and with a foam, which is a natural cobordism  in this context, he associates a graded linear map, with a geometrically understandable degree. This approach allows Khovanov and Kuperberg to show that the web bases are not dual canonical. The goal of this paper is to extend this 1+1 TQFT to a 0+1+1 TQFT. For this purpose we mimic the strategy of \cite{MR1928174} and construct an algebra associated with 0 dimensional objects. The webs are then turned into modules and the foams into morphisms of modules. The same construction is done in \cite{2012arXiv1206.2118M} and they compute the split Grothendieck groups of these algebras. Projective modules associated with non-elliptic webs may be decomposable. We give a large family of webs, called superficial webs, whose associated modules are indecomposable.

\subsubsection*{Acknowledgments} The author wishes to thank Christian Blanchet for suggesting the subject and for plenty of valuable remarks, and Lukas Lewark for fruitful discussions. 
\label{sec:acknoledgment}

\section{A brief review of the algebra $K^\epsilon$}
\label{sec:brief-review-k_s}

\subsection{The 2-category of web-tangles}
\label{sec:webs-foams}

\subsubsection{Webs}
In the following $\epsilon=(\epsilon^1,\dots,\epsilon^n)$ (or $\epsilon_0$, $\epsilon_1$ etc.) will always be a finite sequence of signs, its \emph{length} $n$ will be denoted by $l(\epsilon)$, such an $\epsilon$ will be \emph{admissible} if $\sum_{i=1}^{l(\epsilon)}\epsilon^i$ is divisible by 3.
\label{sec:webs}
\label{sec:two-category-web}
\begin{dfn}[Kuperberg, \cite{MR1403861}]\label{dfn:closed-web}
  A \emph{closed web} is a plane trivalent oriented finite graph (with possibly some vertexless loops and multiple edges) such that every vertex is either a sink or a source.
\end{dfn}
\begin{figure}[h]
  \centering
  \begin{tikzpicture}[yscale= 0.8, xscale= 0.8]
    \begin{scope}
   [yscale = {1}, xscale={1},decoration={markings, mark=at
     position 0.5 with {\arrow{>}}},postaction={decorate}]
\coordinate (A) at (0,0);
\coordinate (B) at (1,0);
\coordinate (C) at (2,-0.5);
\coordinate (D) at (3,0);
\coordinate (E) at (4,0);
\coordinate (F) at (5,0.5);
\coordinate (A1) at (0,1);
\coordinate (B1) at (1,1);
\coordinate (C1) at (2,1.5);
\coordinate (D1) at (3,1);
\coordinate (E1) at (4,1);
\coordinate (F1) at (4,2);
\coordinate (G) at (5,-0.5);
\coordinate (H) at (3.5,-1);

\draw[postaction=decorate] (A1) --(A);
\draw[postaction=decorate] (A1) .. controls +(-0.5,0)  and  +(-0.5,0).. (A);
\draw[postaction=decorate] (A1) -- (B1);
\draw[postaction=decorate] (B)-- (A);
\draw[postaction=decorate] (B)--(B1);
\draw[postaction=decorate] (B)--(C);
\draw[postaction=decorate] (C1)--(B1);
\draw[postaction=decorate] (C1)--(D1);
\draw[postaction=decorate] (C1)--(F1);
\draw[postaction=decorate] (D)--(C);
\draw[postaction=decorate] (D)--(D1);
\draw[postaction=decorate] (D)--(E);
\draw[postaction=decorate] (F) -- (G);
\draw[postaction=decorate] (F)-- (E);
\draw[postaction=decorate] (F) .. controls +(0,0.5) and  +(0.4,0.4) .. (F1);
\draw[postaction=decorate] (E1)--(F1);
\draw[postaction=decorate] (E1)--(D1);
\draw[postaction=decorate] (E1)--(E);
\draw[postaction=decorate] (H)--(C);
\draw[postaction=decorate] (H) .. controls +(0.5,0) and  +(0,-0.5) .. (G);
\draw[postaction=decorate] (H) .. controls +(0.2,0.3) and  +(-0.4,-0.1) .. (G);
\draw[postaction=decorate] (7,1) circle (0.5cm);
\end{scope}
  \end{tikzpicture}  
  \caption{Example of a closed web.}
  \label{fig:example-closed-web}
\end{figure}
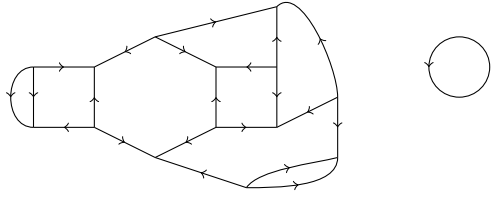
\begin{req}\label{req:basic-on-web}
The orientation condition is equivalent to say that the graph is bipartite (by sinks and sources).  From the Euler characteristic of connected plane graphs, we obtain that a closed-web contains at least a square, a digon or a vertexless circle. 
\end{req}
\begin{prop}[Kuperberg\cite{MR1403861}] \label{prop:Kup}
  There exists one and only one map $\kup{\cdot}$ from closed webs to Laurent polynomials in $q$ which is invariant by isotopy, multiplicative with respect to disjoint union and which satisfies the following local relations~:
  \begin{align*}
   \kup{\websquare[0.4]} &= \kup{\webtwovert[0.4]} + \kup{\webtwohori[0.4]}, \\
   \kup{\webbigon[0.4]\,}  &= [2] \cdot \kup{\webvert[0.4]\,},\\
   \kup{\webcircle[0.4]} &= \kup{\webcirclereverse[0.4]} = [3],
  \end{align*}
where $[n]\eqdef\frac{q^n-q^{-n}}{q-q^{-1}}.$ We call this polynomial the \emph{Kuperberg bracket}. It's easy to check that the Kuperberg bracket of a web is symmetric in $q$ and $q^{-1}$.
\end{prop}
\begin{proof}
  Uniqueness comes from remark \ref{req:basic-on-web}. The existence follows from the representation theoretic point of view developed in \cite{MR1403861}. Note that a non-quantified version of this result is in \cite{MR1172374}.
\end{proof}
\begin{dfn}
  The \emph{degree} of a symmetric Laurent polynomial $P(q)=\sum_{i\in \ZZ}a_iq^{i}$ is  $\max_{i\in \ZZ}\{i \textrm{ such that }a_i\neq 0\}$.
\end{dfn}
\begin{dfn}\label{dfn:webtangle}
  A \emph{$(\epsilon_0,\epsilon_1)$-web-tangle} $w$ is an intersection of a closed web $w'$ with $[0,1]\times [0,1]$ such that :
  \begin{itemize}
  \item there exists $\eta_0 \in ]0,1]$ such that $w\cap [0,1] \times [0,\eta_0] = \{\frac{1}{2l(\epsilon_0)}, \frac{1}{2l(\epsilon_0)} + \frac{1}{l(\epsilon_0)},  \frac{1}{2l(\epsilon_0)} + \frac{2}{l(\epsilon_0)}, \dots, \frac{1}{2l(\epsilon_0)} + \frac{l(\epsilon_0)-1}{l(\epsilon_0)} \}\times [0,\eta_0]$,
  \item there exists $\eta_1 \in [0,1[$ such that $w\cap [0,1] \times [\eta_1,1] = \{\frac{1}{2l(\epsilon_1)}, \frac{1}{2l(\epsilon_1)} + \frac{1}{l(\epsilon_1)},  \frac{1}{2l(\epsilon_1)} + \frac{2}{l(\epsilon_1)}, \dots, \frac{1}{2l(\epsilon_1)} + \frac{l(\epsilon_1)-1}{l(\epsilon_1)} \}\times [\eta_1,1]$,
  \item the orientations of the edges of $w$, match $-\epsilon_0$ and $+\epsilon_1$ (see figure \ref{fig:exampl_webtangle} for conventions).
  \end{itemize}
When $\epsilon_1$ is the empty sequence, then we'll speak of \emph{$\epsilon_0$-webs}. And if $w$ is an $\epsilon$-web we will say that $\epsilon$ is the \emph{boundary} of $w$.
\end{dfn}
If $w_1$ is a $(\epsilon_0,\epsilon_1)$-web-tangle and $w_2$ is a $(\epsilon_1, \epsilon_2)$-web-tangle we define $w_1w_2$ to be the $(\epsilon_0,\epsilon_2)$-web-tangle obtained by gluing $w_1$ and $w_2$ along $\epsilon_1$ and resizing. Note that this can be thought as a composition if we think about a $(\epsilon,\epsilon')$-web-tangle as a morphism from $\epsilon'$ to $\epsilon$ (\ie the web-tangles should be read as morphisms from top to bottom). The \emph{mirror image} of a $(\epsilon_0,\epsilon_1)$-web-tangle $w$ is mirror image of $w$ with respect to $\RR\times \{\frac12\}$ with all orientations reversed. This is a $(\epsilon_1,\epsilon_0)$-web-tangle and we denote it by $\bar{w}$. If $w$ is a $(\epsilon,\epsilon)$-web-tangle the \emph{closure} of $w$ is the closed web obtained by connecting the top and the bottom by simple arcs (this is like a braid closure). We denote it by $\mathrm{tr}(w)$.   
\begin{figure}[h]
  \centering
  \begin{tikzpicture}[yscale= 0.5, xscale= 0.5]
    \input{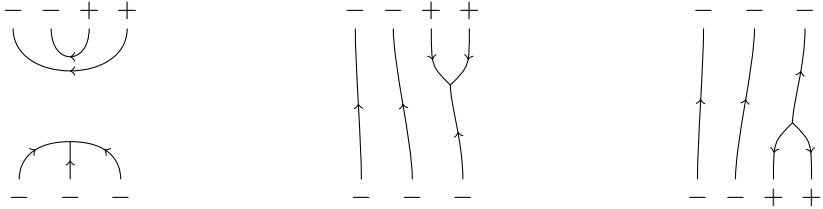}
  \end{tikzpicture}
  \caption{Two examples of $(\epsilon_0,\epsilon_1)$-web-tangles with $\epsilon_0 = (-,-,-)$ and $\epsilon_1=(-,-,+,+)$ and the mirror image of the second one.}
  \label{fig:exampl_webtangle}
\end{figure}
\begin{dfn}
  A web-tangle with no circle, no digon and no square is said to be \emph{non-elliptic}. The non-elliptic web-tangles are the minimal ones in the sense that they cannot be reduced by the relations of proposition \ref{prop:Kup}.
\end{dfn}
\begin{prop}[Kuperberg, \cite{MR1403861}] \label{prop:NEFinite}
  For any given couple $(\epsilon_0,\epsilon_1)$ of sequences of signs there are finitely many non-elliptic $(\epsilon_0,\epsilon_1)$-web-tangles.
\end{prop}
\begin{req}
  From the combinatorial flow modulo 3, we obtain that there exist some $(\epsilon_0,\epsilon_1)$-webs if and only if the sequence $-\epsilon_0$ concatenated with $\epsilon_1$ is admissible.
\end{req}
\subsubsection{Foams} All material here comes from \cite{MR2100691}.
\label{sec:foams}
\begin{dfn}
  A \emph{pre-foam} is a smooth oriented compact surface $\Sigma$ (its connected component are called \emph{facets}) together with the following data~:
\begin{itemize}
\item A partition of the connected components of the boundary into cyclically ordered 3-sets and for each 3-set $(C_1,C_2,C_3)$, three orientation preserving diffeomorphisms $\phi_1:C_2\to C_3$, $\phi_2:C_3\to C_1$ and $\phi_3:C_1\to C_2$ such that $\phi_3 \circ \phi_2 \circ \phi_1 = \mathrm{id}_{C_2}$.
\item A function from the set of facets to the set of non-negative integers (this gives the number of \emph{dots} on each facet).
\end{itemize}
The \emph{CW-complex associated with a pre-foam} is the 2-dimensional CW-complex $\Sigma$ quotiented by the diffeomorphisms so that the three circles of one 3-set are identified and become just one called a \emph{singular circle}.
The \emph{degree} of a pre-foam $f$ is equal to $-2\chi(\Sigma')$ where $\chi$ is the Euler characteristic, $\Sigma'$ is the CW-complex associated with $f$ with the dots punctured out (\ie a dot increases the degree by 2).
\end{dfn}
\begin{req}
  The CW-complex has two local models depending on whether we are on a singular circle or not. If a point $x$ is not on a singular circle, then it has a neighborhood diffeomorphic to a 2-dimensional disk, else it has a neighborhood diffeomorphic to a Y shape times an interval (see figure \ref{fig:yshape}).
  \begin{figure}[h]
    \centering
    \begin{tikzpicture}[scale=1]
      \input{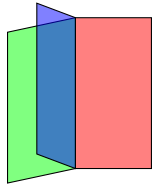}
    \end{tikzpicture}
    \caption{Singularities of a pre-foam}
    \label{fig:yshape}
  \end{figure}
\end{req}
\begin{dfn}
  A \emph{closed foam} is the image of an embedding of the CW-complex associated with a pre-foam such that the cyclic orders of the pre-foam are compatible with the left-hand rule in $\RR^3$ with respect to the orientation of the singular circles\footnote{We mean here that if, next to a singular circle, with the forefinger of the left hand we go from face 1 to face 2 to face 3 the thumb points to indicate the orientation of the singular circle (induced by orientations of facets). This is not quite canonical, physicists use more the right-hand rule, however this is the convention used in \cite{MR2100691}.}. The \emph{degree} of a closed foam is the degree of the underlying pre-foam. 
\end{dfn}
\begin{dfn}\label{dfn:wwfoam}
  If $w_b$ and $w_t$ are $(\epsilon_0,\epsilon_1)$-web-tangles, a \emph{$(w_b,w_t)$-foam} $f$ is the intersection of a foam $f'$ with $\RR\times [0,1]\times[0,1]$ such that
  \begin{itemize}
  \item there exists $\eta_0 \in ]0,1]$ such that $f\cap \RR \times [0,\eta_0]\times [0,1] = \{\frac{1}{2l(\epsilon_0)}, \frac{1}{2l(\epsilon_0)} + \frac{1}{l(\epsilon_0)},  \frac{1}{2l(\epsilon_0)} + \frac{2}{l(\epsilon_0)}, \dots, \frac{1}{2l(\epsilon_0)} + \frac{l(\epsilon_0)-1}{l(\epsilon_0)} \}\times [0,\eta_0]\times [0,1]$,
  \item there exists $\eta_1 \in [0,1[$ such that $f\cap \RR \times [\eta_1,1]\times [0,1] = \{\frac{1}{2l(\epsilon_1)}, \frac{1}{2l(\epsilon_1)} + \frac{1}{l(\epsilon_1)},  \frac{1}{2l(\epsilon_1)} + \frac{2}{l(\epsilon_1)}, \dots, \frac{1}{2l(\epsilon_1)} + \frac{l(\epsilon_1)-1}{l(\epsilon_1)}\}\times [\eta_1,1]\times [0,1]$,
  \item there exists $\eta_b \in ]0,1]$ such that $f\cap \RR \times [0,1 ]\times  [0, \eta_b] = w_b \times [0, \eta_b]$,
  \item there exists $\eta_t \in [0,1[$ such that $f\cap \RR \times [0,1 ]\times  [\eta_t, 1] = w_t \times [\eta_t,1]$,
  \end{itemize}
with compatibility of orientations of the facets of $f$ with the orientation of $w_t$ and the reversed orientation of $w_b$.
The \emph{degree} of a $(w_b,w_t)$-foam $f$ is equal to $\chi(w_b)+\chi(w_t)-2\chi(\Sigma)$ where $\Sigma$ is the underlying CW-complex associated with $f$ with the dots punctured out. 
\end{dfn}
If $f_b$ is a $(w_b,w_m)$-foam and $f_t$ is a $(w_m, w_t)$-foam we define $f_bf_t$ to be the $(w_b,w_t)$ foam obtained by gluing $f_b$ and $f_t$ along $w_m$ and resizing. This operation may be thought as a composition if we think of a $(w_1,w_2)$-foam as a morphism from $w_2$ to $w_1$ \ie from the top to the bottom. This composition map is degree preserving. Like for the webs, we define the \emph{mirror image} of a $(w_1,w_2)$-foam $f$ to be the $(w_2,w_1)$-foam which is the mirror image of $f$ with respect to $\RR\times\RR\times \{\frac12\}$ with all orientations reversed. We denote it by $\bar{f}$.
\begin{dfn}
  If $\epsilon_0=\epsilon_1=\emptyset$ and $w$ is a closed web, then a $(\emptyset,w)$-foam is simply called \emph{foam} or \emph{$w$-foam} when one wants to focus on the boundary of the foam.
\end{dfn}
All these data together lead  to the definition of a monoidal 2-category. 
\begin{dfn}
  The 2-category $\mathcal{WT}$ is the monoidal\footnote{Here we choose a rather strict point of view and hence the monoidal structure is strict (we consider everything up to isotopy), but it is possible to define the notion in a non-strict context, and the same data gives us a monoidal bicategory.} 2-category given by the following data~:
  \begin{itemize}
  \item The objects are finite sequences of signs,
  \item The 1-morphism from $\epsilon_1$ to $\epsilon_0$ are isotopy classes (with fixed boundary) of $(\epsilon_0,\epsilon_1)$-web-tangles,
  \item The 2-morphism from $\widehat{w_t}$ to $\widehat{w_b}$ are $\QQ$-linear combinations of isotopy classes of $(w_b,w_t)$-foams, where\ $\widehat{\cdot}$\ stands for the ``isotopy class of''. The 2-morphisms come with a grading, the composition respects the degree.
  \end{itemize}
The monoidal structure is given by concatenation of sequences at the $0$-level, and disjoint union of vertical strands or disks (with corners) at the $1$ and $2$ levels. 
\end{dfn}
\subsection{Khovanov's TQFT for web-tangles}
\label{sec:khovanov-tqft-web}
In \cite{MR2100691}, Khovanov defines a numerical invariant for pre-foams and this allows him to construct a TQFT $\mathcal{F}$ from the category $\hom_{\mathcal{WT}}(\emptyset,\emptyset)$ to the category of graded $\QQ$-modules, (via a universal construction à la BHMV 
\cite{MR1362791}). This TQFT is graded (this comes from the fact that pre-foams with non-zero degree are evaluated to zero), and satisfies the following local relations (brackets indicate grading shifts)~:
\begin{align*}
%   \dim_q 
\mathcal{F}\left(\websquare[0.4]\,\right) &= 
%\dim_q 
\mathcal{F}\left({\webtwovert[0.4]}\,\right) \oplus
%\dim_q 
\mathcal{F}\left({\webtwohori[0.4]}\,\right), \\
%   \dim_q 
\mathcal{F}\left({\webbigon[0.4]}\,\right)  &= 
%\dim_q 
\mathcal{F}\left({\webvert[0.4]}\,\right)\{-1\}\oplus \mathcal{F}\left({\webvert[0.4]}\right)\{1\},\\
%   \dim_q 
\mathcal{F}\left({\webcircle[0.4]\,}\right) &= 
%\dim_q 
\mathcal{F}\left({\webcirclereverse[0.4]}\,\right) = \QQ\{-2\} \oplus \QQ \oplus\QQ\{2\}.
  \end{align*}
These relations show that $\mathcal{F}$ is a categorified counterpart of the Kuperberg bracket. We sketch the construction below.
\begin{dfn}
We denote by $\mathcal{A}$ the Frobenius algebra $\ZZ[X]/(X^3)$ with trace $\tau$ given by:
\[\tau(X^2)=-1, \quad \tau(X)=0, \quad \tau(1)=0.\] 
We equip $\mathcal{A}$ with a graduation by setting $\deg(1)=-2$, $\deg(X)=0$ and $\deg(X^2)=2$. With these settings, the multiplication has degree 2 and the trace has degree -2. The co-multiplication is determined by the multiplication and the trace and we have :
\begin{align*}
  &\Delta(1) = -1\otimes X^2 - X\otimes X - X^2\otimes 1 \\
  &\Delta(X) = -X\otimes X^2 - X^2\otimes X \\
  &\Delta(X^2) = -X^2\otimes X^2
\end{align*}
\end{dfn}
This Frobenius algebra gives us a 1+1 TQFT (this is well-known, see \cite{kock2004frobenius} for details), we denote it by $\mathcal{F}$~: the circle is sent to $\mathcal{A}$, a cup to the unity, a cap to the trace, and a pair of pants either to multiplication or co-multiplication. A dot on a surface will represent multiplication by $X$ so that $\mathcal{F}$ extends to the category of oriented dotted (1+1)-cobordisms.
We then have a surgery formula given by figure~\ref{fig:surg}. 
\begin{figure}[h!]
  \centering
  \begin{tikzpicture}[scale=0.5]
    \begin{scope}[xshift=-0.5cm]
       \draw (0,-1) arc (270:90:0.5 and 1);
      \draw[dotted] (0,-1) arc (-90:90:0.5 and 1);
      \draw (3,0) ellipse (0.5 and 1);
      \draw (0,-1) -- (3,-1);
      \draw (0,1) -- (3,1);
\end{scope}
\node at (4.5,0) {$=-$};
\node at (10,0) {$-$};
\node at (15,0) {$-$};
\begin{scope}[xshift=6cm]
       \draw (0,-1) arc (270:90:0.5 and 1);
      \draw[dotted] (0,-1) arc (-90:90:0.5 and 1);
      \draw (3,0) ellipse (0.5 and 1);
      \draw (0,-1) .. controls +(1.5,0) and +(1.5,0) .. (0,1);
      \draw (3,-1) .. controls +(-1.5,0) and +(-1.5,0) .. (3,1);
      \filldraw (0.7,0.2) ellipse (1pt and 2pt);
      \filldraw (0.7,-0.2) ellipse (1pt and 2pt);
\end{scope}
\begin{scope}[xshift=11cm]
       \draw (0,-1) arc (270:90:0.5 and 1);
      \draw[dotted] (0,-1) arc (-90:90:0.5 and 1);
      \draw (3,0) ellipse (0.5 and 1);
      \draw (0,-1) .. controls +(1.5,0) and +(1.5,0) .. (0,1);
      \draw (3,-1) .. controls +(-1.5,0) and +(-1.5,0) .. (3,1);
      \filldraw (0.7,0) ellipse (1pt and 2pt);
      \filldraw (2.2,0) ellipse (1pt and 2pt);
\end{scope}
\begin{scope}[xshift=16cm]
       \draw (0,-1) arc (270:90:0.5 and 1);
      \draw[dotted] (0,-1) arc (-90:90:0.5 and 1);
      \draw (3,0) ellipse (0.5 and 1);
      \draw (0,-1) .. controls +(1.5,0) and +(1.5,0) .. (0,1);
      \draw (3,-1) .. controls +(-1.5,0) and +(-1.5,0) .. (3,1);
      \filldraw (2.2,-0.2) ellipse (1pt and 2pt);
      \filldraw (2.2,0.2) ellipse (1pt and 2pt);
\end{scope}
  \end{tikzpicture}
  \caption{The surgery formula for the TQFT $\mathcal{F}$.}
  \label{fig:surg}
\end{figure}
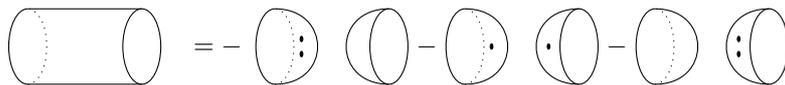

 This TQFT gives of course a numercial invariant for closed dotted oriented surfaces. If one defines numerical values for the differently dotted theta pre-foams (the theta pre-foam consists of 3 disks with trivial diffeomorphisms between their boundary see figure \ref{fig:thetapre}) then by applying the surgery formula, one is able to compute a numerical value for pre-foams. 

\begin{figure}[h!]
  \centering
  \begin{tikzpicture}
    \input{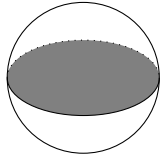}
  \end{tikzpicture}
  \caption{The dotless theta pre-foam.}
  \label{fig:thetapre}
\end{figure}
\begin{figure}[h!]
  \centering
\begin{tikzpicture}[scale = 0.7]
\input{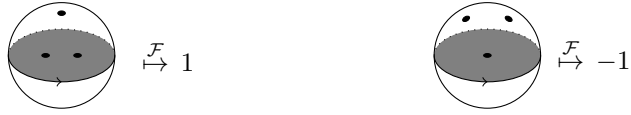}
\end{tikzpicture}
  \caption{The evaluations of dotted theta foams, the evaluation is unchanged when one cyclically permutes the faces. All the configurations which cannot be obtained from these by cyclic permutation are sent to $0$ by $\mathcal{F}$.}
  \label{fig:thetaeval}
\end{figure}

In \cite{MR2100691}, Khovanov shows that setting the evaluations of the dotted theta foams as shown on figure \ref{fig:thetaeval}, leads to a well defined numerical invariant $\mathcal{F}$ for pre-foams. This numerical invariant gives the opportunity to build a (closed web, $(\cdot,\cdot)$-foams)-TQFT~: for a given web $w$, consider the $\QQ$-module generated by all the $(w,\emptyset)$-foams, and mod this space out by the kernel of the bilinear map $(f,g)\mapsto \mathcal{F}(\bar{f}g)$. Note that $\bar{f}g$ is a closed foam. Khovanov showed that the obtained graded vector spaces are finite dimensional with graded dimensions given by the Kuperberg formulae, and he showed that we have the local relations described on figure~\ref{fig:localrel}.

\begin{figure}[h]
  \centering 
  \begin{tikzpicture}[scale=0.25]
\begin{scope}[xshift=0cm, yshift= 0cm, decoration={markings, mark=at
     position 0.5 with {\arrow{>}}},postaction={decorate}]
\draw[postaction= {decorate}] (-2,0) arc (-180:0:2 and 1);
\draw (-2,0) arc (180:0:2);
\draw[dashed] (-2,0) arc (180:0:2 and 1);
\draw[dashed] (-2,0) arc (-180:-135:2);
\draw[dashed] (2,0) arc (0:-45:2);
\draw (0,0)++(-135:2) arc (-135:-45:2);
\draw[dashed] (0,0) +(135:2) -- +(45:2);
\draw (0,0)++(135:2) --  ++(-0.7,0) -- ++(-135:4) -- ++(7,0) -- ++(45:4) -- (45:2);
%\draw[dashed] (0,0) ++(120:2) -- +(2,0);
\filldraw (0,1.8) ellipse (3pt and 1.5pt);
\end{scope}
\node at (5,0) {$=-$};
\begin{scope}[xshift=10cm, yshift= 0cm, decoration={markings, mark=at
     position 0.5 with {\arrow{>}}},postaction={decorate}]
\draw[postaction= {decorate}] (-2,0) arc (-180:0:2 and 1);
\draw (-2,0) arc (180:0:2);
\draw[dashed] (-2,0) arc (180:0:2 and 1);
\draw[dashed] (-2,0) arc (-180:-135:2);
\draw[dashed] (2,0) arc (0:-45:2);
\draw (0,0)++(-135:2) arc (-135:-45:2);
\draw[dashed] (0,0) +(135:2) -- +(45:2);
\draw (0,0)++(135:2) --  ++(-0.7,0) -- ++(-135:4) -- ++(7,0) -- ++(45:4) -- (45:2);
%\draw[dashed] (0,0) ++(120:2) -- +(2,0);
%\filldraw (0,1.8) ellipse (3pt and 1.5pt);
\filldraw (0,-1.8) ellipse (3pt and 1.5pt);
\end{scope}
\node at (15,0) {$=$};
\begin{scope}[xshift=20cm, yshift= 0cm, decoration={markings, mark=at
     position 0.5 with {\arrow{>}}},postaction={decorate}]
\draw (0,0)++(135:2) --  ++(-0.7,0) -- ++(-135:4) -- ++(7,0) -- ++(45:4) -- (135:2);
%\filldraw (0,0) ellipse (3pt and 1.5pt);
\end{scope}
\begin{scope} [yshift= -6cm]
  \begin{scope}[xshift=0cm, yshift= 0cm, decoration={markings, mark=at
     position 0.5 with {\arrow{>}}},postaction={decorate}]
\draw[postaction= {decorate}] (-2,0) arc (-180:0:2 and 1);
\draw (-2,0) arc (180:0:2);
\draw[dashed] (-2,0) arc (180:0:2 and 1);
\draw[dashed] (-2,0) arc (-180:-135:2);
\draw[dashed] (2,0) arc (0:-45:2);
\draw (0,0)++(-135:2) arc (-135:-45:2);
\draw[dashed] (0,0) +(135:2) -- +(45:2);
\draw (0,0)++(135:2) --  ++(-0.7,0) -- ++(-135:4) -- ++(7,0) -- ++(45:4) -- (45:2);
%\draw[dashed] (0,0) ++(120:2) -- +(2,0);
%\filldraw (0,1.8) ellipse (3pt and 1.5pt);
\end{scope}
\node at (5,0) {$=$};
\begin{scope}[xshift=10cm, yshift= 0cm, decoration={markings, mark=at
     position 0.5 with {\arrow{>}}},postaction={decorate}]
\draw[postaction= {decorate}] (-2,0) arc (-180:0:2 and 1);
\draw (-2,0) arc (180:0:2);
\draw[dashed] (-2,0) arc (180:0:2 and 1);
\draw[dashed] (-2,0) arc (-180:-135:2);
\draw[dashed] (2,0) arc (0:-45:2);
\draw (0,0)++(-135:2) arc (-135:-45:2);
\draw[dashed] (0,0) +(135:2) -- +(45:2);
\draw (0,0)++(135:2) --  ++(-0.7,0) -- ++(-135:4) -- ++(7,0) -- ++(45:4) -- (45:2);
%\draw[dashed] (0,0) ++(120:2) -- +(2,0);
\filldraw (0,1.8) ellipse (3pt and 1.5pt);
\filldraw (0,-1.8) ellipse (3pt and 1.5pt);
\end{scope}
\node at (15,0) {$=0$};
%\begin{scope}[xshift=20cm, yshift= 0cm, decoration={markings, mark=at
%     position 0.5 with {\arrow{>}}},postaction={decorate}]
%\draw (0,0)++(135:2) --  ++(-0.7,0) -- ++(-135:4) -- ++(7,0) -- ++(45:4) -- (135:2);
%\filldraw (0,0) ellipse (3pt and 1.5pt);
%\end{scope}
\end{scope}
\begin{scope}[yshift=-12cm]
  \begin{scope}[xshift=0cm, yshift= 0cm, decoration={markings, mark=at
     position 0.5 with {\arrow{>}}},postaction={decorate}]
\draw[postaction= {decorate}] (-2,0) arc (-180:0:2 and 1);
\draw (-2,0) arc (180:0:2);
\draw[dashed] (-2,0) arc (180:0:2 and 1);
\draw[dashed] (-2,0) arc (-180:-135:2);
\draw[dashed] (2,0) arc (0:-45:2);
\draw (0,0)++(-135:2) arc (-135:-45:2);
\draw[dashed] (0,0) +(135:2) -- +(45:2);
\draw (0,0)++(135:2) --  ++(-0.7,0) -- ++(-135:4) -- ++(7,0) -- ++(45:4) -- (45:2);
%\draw[dashed] (0,0) ++(120:2) -- +(2,0);
\filldraw (-0.2,-1.8) ellipse (3pt and 1.5pt);
\filldraw (0.2,-1.8) ellipse (3pt and 1.5pt);
\end{scope}
\node at (5,0) {$=-$};
\begin{scope}[xshift=10cm, yshift= 0cm, decoration={markings, mark=at
     position 0.5 with {\arrow{>}}},postaction={decorate}]
\draw[postaction= {decorate}] (-2,0) arc (-180:0:2 and 1);
\draw (-2,0) arc (180:0:2);
\draw[dashed] (-2,0) arc (180:0:2 and 1);
\draw[dashed] (-2,0) arc (-180:-135:2);
\draw[dashed] (2,0) arc (0:-45:2);
\draw (0,0)++(-135:2) arc (-135:-45:2);
\draw[dashed] (0,0) +(135:2) -- +(45:2);
\draw (0,0)++(135:2) --  ++(-0.7,0) -- ++(-135:4) -- ++(7,0) -- ++(45:4) -- (45:2);
%\draw[dashed] (0,0) ++(120:2) -- +(2,0);
\filldraw (0.2,1.8) ellipse (3pt and 1.5pt);
\filldraw (-0.2,1.8) ellipse (3pt and 1.5pt);
\end{scope}
\node at (15,0) {$=$};
\begin{scope}[xshift=20cm, yshift= 0cm, decoration={markings, mark=at
     position 0.5 with {\arrow{>}}},postaction={decorate}]
\draw (0,0)++(135:2) --  ++(-0.7,0) -- ++(-135:4) -- ++(7,0) -- ++(45:4) -- (135:2);
\filldraw (0,0) ellipse (3pt and 1.5pt);
\end{scope}
\end{scope}
\end{tikzpicture}

\vspace{0.7cm}
\begin{tikzpicture}[scale=0.24]
\begin{scope}[xshift=0cm, yshift= 0cm, decoration={markings, mark=at
     position 0.5 with {\arrow{>}}},postaction={decorate}]
\draw (0,-2) arc (270:90:1 and 2);
\draw[dashed] (0,-2) arc (-90:90:1 and 2);
\fill[gray] (3,0) ellipse (1 and 2);
\draw (3,-2) arc (270:90:1 and 2);
\draw[dashed] (3,-2) arc (-90:90:1 and 2);
\fill[gray] (6,0) ellipse (1 and 2);
\draw (6,-2) arc (270:90:1 and 2);
\draw[dashed] (6,-2) arc (-90:90:1 and 2);
\draw (9,0) ellipse (1 and 2);
\draw (0,-2) -- (9,-2);
\draw (0,2) -- (9,2);
\node at (12,0) {$=-$};
\draw (15,-2) arc (270:90:1 and 2);
\draw[dashed] (15,-2) arc (-90:90:1 and 2);
\draw (15,-2) arc (-90:90:2 and 2);
\draw (20,-2) arc (-90:-270:2 and 2);
\draw(20,0) ellipse (1 and 2);
\end{scope}

\begin{scope}[xshift=-2cm, yshift= -6cm, decoration={markings, mark=at
     position 0.5 with {\arrow{>}}},postaction={decorate}]
\draw (0,-2) arc (270:90:1 and 2);
\draw[dashed] (0,-2) arc (-90:90:1 and 2);
\fill[gray] (3,0) ellipse (1 and 2);
\draw[postaction= {decorate}] (3,-2) arc (270:90:1 and 2);
\draw[dashed] (3,-2) arc (-90:90:1 and 2);
%\draw (6,-2) arc (270:90:1 and 2);
%\draw[dashed] (6,-2) arc (-90:90:1 and 2);
\draw (6,0) ellipse (1 and 2);
\draw (0,-2) -- (6,-2);
\draw (0,2) -- (6,2);
\node at (8,0) {$=$};
\draw (10,-2) arc (270:90:1 and 2);
\draw[dashed] (10,-2) arc (-90:90:1 and 2);
\draw (10,-2) arc (-90:90:2 and 2);
\draw (15,-2) arc (-90:-270:2 and 2);
\draw(15,0) ellipse (1 and 2);
\fill (15,0) ellipse (2pt and 4pt);
\node at (17,0) {$-$};
\fill (19,0) ellipse (2pt and 4pt);b
\draw (19,-2) arc (270:90:1 and 2);
\draw[dashed] (19,-2) arc (-90:90:1 and 2);
\draw (19,-2) arc (-90:90:2 and 2);
\draw (24,-2) arc (-90:-270:2 and 2);
\draw(24,0) ellipse (1 and 2);
\end{scope} 
\end{tikzpicture}

\vspace{0.7cm}
\begin{tikzpicture}[scale=0.4]
\begin{scope}[xshift=0cm, yshift= 0cm, decoration={markings, mark=at
     position 0.5 with {\arrow{>}}},postaction={decorate}]
\draw (-2,0) -- +(1,0);
\draw (1,0) -- +(1,0);
\draw[dashed] (-1,0) .. controls +(1,0.5) and +(-1,0.5).. +(2,0);
\draw (-1,0) .. controls +(1,-0.5) and +(-1,-0.5).. +(2,0);
\draw (-2,3) -- +(1,0);
\draw (1,3) -- +(1,0);
\draw (-1,3) .. controls +(1,0.5) and +(-1,0.5).. +(2,0);
\draw (-1,3) .. controls +(1,-0.5) and +(-1,-0.5).. +(2,0);
\draw (-2,0) -- +(0,3);
\draw[postaction={decorate}] (-1,0) -- +(0,3);
\draw[postaction={decorate}] (1,3) -- +(0,-3);
\draw (2,0) -- +(0,3);
\end{scope}
\node at (3,1.5) {$=$};
\begin{scope}[xshift=6cm, yshift= 0cm, decoration={markings, mark=at
     position 0.5 with {\arrow{>}}},postaction={decorate}]
\draw (-2,0) -- +(1,0);
\draw (1,0) -- +(1,0);
\draw[dashed] (-1,0) ..controls +(1,0.5) and +(-1,0.5).. +(2,0);
\draw (-1,0) ..controls +(1,-0.5) and +(-1,-0.5).. +(2,0);
\draw (-2,3) -- +(1,0);
\draw (1,3) -- +(1,0);
\draw (-1,3) ..controls +(1,0.5) and +(-1,0.5).. +(2,0);
\draw (-1,3) ..controls +(1,-0.5) and +(-1,-0.5).. +(2,0);
\draw (-2,0) -- +(0,3);
\draw[postaction={decorate}] (-1,0) .. controls +(0,1.5) and +(0,1.5).. +(2,0);
\draw[postaction={decorate}] (1,3) .. controls +(0,-1.5) and +(0,-1.5).. +(-2,0);
\draw (2,0) -- +(0,3);
\fill (0,3) circle (2pt and 2pt);
\end{scope}
\node at (9,1.5) {$-$}; 
\begin{scope}[xshift=12cm, yshift= 0cm, decoration={markings, mark=at
     position 0.5 with {\arrow{>}}},postaction={decorate}]
\draw (-2,0) -- +(1,0);
\draw (1,0) -- +(1,0);
\draw[dashed] (-1,0).. controls +(1,0.5) and +(-1,0.5).. +(2,0);
\draw (-1,0) ..controls +(1,-0.5) and +(-1,-0.5).. +(2,0);
\draw (-2,3) -- +(1,0);
\draw (1,3) -- +(1,0);
\draw (-1,3).. controls +(1,0.5) and +(-1,0.5).. +(2,0);
\draw (-1,3).. controls +(1,-0.5) and +(-1,-0.5).. +(2,0);
\draw (-2,0) -- +(0,3);
\draw[postaction={decorate}] (-1,0) .. controls +(0,1.5) and +(0,1.5).. +(2,0);
\draw[postaction={decorate}] (1,3) .. controls +(0,-1.5) and +(0,-1.5).. +(-2,0);
\draw (2,0) -- +(0,3);
\fill (0,0) circle (2pt and 2pt);
\end{scope} 
\end{tikzpicture}

\vspace{0.7cm}
\begin{tikzpicture}[scale=0.32]
\begin{scope}[xshift=0cm, yshift= 0cm]
\draw (0,0) -- (2,0);
\draw (2,0) -- (1,1);
\draw (1,1) -- (-1,1); 
\draw (-1,1) -- (0,0);
\draw (0,0) -- +(-0.5,-0.5);
\draw (2,0) -- +(0.7,-0.3);
\draw (1,1) -- +(0.5,0.5); 
\draw (-1,1) -- +(-0.7,0.3);
\draw (0,-4) -- (2,-4);
\draw[dashed] (2,-4) -- (1,-3);
\draw[dashed] (1,-3) -- (-1,-3); 
\draw (-1,-3) -- (-0.5,-3.5);
\draw[dashed] (-0.5,-3.5) -- (0,-4);
\draw (0,-4) -- ++(-0.5,-0.5) -- +(0,4);
\draw (2,-4) -- ++(0.7,-0.3)--+(0,4);
\draw[dashed] (1,-3) -- ++(0.5,0.5)--+(0,3);
\draw (1.5,0.5) -- (1.5,1.5); 
\draw (-1,-3) -- ++(-0.7,0.3)--+(0,4);
\draw (0,-4) -- +(0,4);
\draw (2,-4) -- +(0,4);
\draw[dashed] (1,-3) -- +(0,3);
\draw (1,0)-- +(0,1);
\draw (-1,-3) -- +(0,4);
\end{scope}
\node at (3.9,-2) {$=-$};
\begin{scope}[xshift=7cm, yshift= 0cm, decoration={markings, mark=at
     position 0.5 with {\arrow{>}}},postaction={decorate}]
\fill[gray, opacity =0.5] (-1,1) -- (0,0) arc (180:225:1) -- +(-1,1) arc (225:180:1)-- cycle;
\fill[gray,opacity =0.5] (1,1) -- (2,0) arc (0:-135:1) -- +(-1,1) arc(-135:0:1) --cycle;
\fill[gray,opacity=0.5] (-1,-3) -- (0,-4) arc (180:45:1) -- +(-1,1) arc (45:180:1)-- cycle;
\fill[gray,opacity=0.5] (1,-3) -- (2,-4) arc (0:45:1) -- +(-1,1) arc(45:0:1) --cycle;
%\fill[white, opacity=0.5] (2,0) arc (0:-180:1);
\draw (0,0) -- (2,0);
\draw (2,0) -- (1,1);
\draw (1,1) -- (-1,1); 
\draw (-1,1) -- (0,0);
\draw (0,0) -- +(-0.5,-0.5);
\draw (2,0) -- +(0.7,-0.3);
\draw (1,1) -- +(0.5,0.5); 
\draw (-1,1) -- +(-0.7,0.3);
\draw (0,-4) -- (2,-4);
\draw[dashed] (2,-4) -- (1,-3);
\draw[dashed] (1,-3) -- (-1,-3); 
\draw (-1,-3) -- (-0.5,-3.5);
\draw[dashed] (-0.5,-3.5) -- (0,-4);
\draw (0,-4) -- ++(-0.5,-0.5) -- +(0,4);
\draw (2,-4) -- ++(0.7,-0.3)--+(0,4);
\draw[dashed] (1,-3) -- ++(0.15,0.15);
\draw (1,-3) + (0.15,0.15) --++(0.5,0.5) -- +(0,1.7);
\draw[dashed] (1.5,0.5) -- +(0,-1.3);
\draw (1.5,0.5) -- (1.5,1.5); 
\draw (-1,-3) -- ++(-0.7,0.3)--+(0,4);
\draw (0,-4) arc (180:0:1);
\draw (0,0) arc (-180:0:1);
%\draw (-1,1) arc (-180:-90:1);
\draw[dashed] (1,1) arc (0:-180:1);
\draw (-1,-3) arc (180:120:1);
\draw[dashed] (1,-3) arc (0:120:1);
\draw (0,-3)++(45:1) -- +(1,-1);
\draw (0,1)++(-135:1) -- +(1,-1);
%\draw (0,-4) .. controls +(0,1.5) and +(0,1.5) .. (2,-4);
%\draw (0,0) .. controls +(0,-1.5) and +(0,-1.5) .. (2,0);
%\draw (0,-4) -- +(0,4);
%\draw (2,-4) -- +(0,4);
%\draw[dashed] (1,-3) -- +(0,3);
%\draw (1,0)-- +(0,1);
%\draw (-1,-3) -- +(0,4);
\end{scope}
\node at (10.9,-2) {$-$};
\begin{scope}[xshift=14cm, yshift= 0cm, decoration={markings, mark=at
     position 0.5 with {\arrow{>}}},postaction={decorate}]
\fill[gray,opacity=0.5] (-1,1) ..controls +(0,-1.5) and +(-0.1,0).. (-0.5,-0.8) -- ++(2,0).. controls +(-0.1,0) and +(0,-1.5) .. (1,1);
\fill[gray,opacity=0.5] (0,0) ..controls +(0,-0.8) and +(0.1,0).. (-0.5,-0.8) -- ++(2,0).. controls +(0.1,0) and +(0,-0.8) .. (2,0);
\fill[gray,opacity=0.5] (-1,-3) ..controls +(0,0.8) and +(-0.1,0).. ++(0.5,0.8) -- ++(2,0).. controls +(-0.1,0) and +(0,0.8) .. (1,-3);
\fill[gray,opacity=0.5] (0,-4) ..controls +(0,1.5) and +(0.1,0).. (-0.5,-2.2) -- ++(2,0).. controls +(0.1,0) and +(0,1.5) .. (2,-4);
\draw (0,0) -- (2,0);
\draw (2,0) -- (1,1);
\draw (1,1) -- (-1,1); 
\draw (-1,1) -- (0,0);
\draw (0,0) -- +(-0.5,-0.5);
\draw (2,0) -- +(0.7,-0.3);
\draw (1,1) -- +(0.5,0.5); 
\draw (-1,1) -- +(-0.7,0.3);
\draw (0,-4) -- (2,-4);
\draw[dashed] (2,-4) -- (1,-3);
\draw[dashed] (1,-3) -- (-1,-3); 
\draw (-1,-3) -- (-0.5,-3.5);
\draw[dashed] (-0.5,-3.5) -- (0,-4);
\draw (0,-4) -- ++(-0.5,-0.5) -- +(0,4);
\draw (2,-4) -- ++(0.7,-0.3)--+(0,4);
\draw[dashed] (1,-3) -- ++(0.5,0.5)--+(0,0.3);
\draw (1.5,-2.2) -- (1.5,-0.8);
\draw[dashed] (1.5,-0.8) -- (1.5,0.5);
\draw (1.5,0.5)-- (1.5,1.5 );
\draw (1.5,0.5) -- (1.5,1.5); 
\draw (-1,-3) -- ++(-0.7,0.3)--+(0,4);
\draw[dashed] (1,-3) ..controls +(0,0.8) and +(-0.1,0) .. ++(0.5,0.8).. controls +(0.1,0) and +(0,1.5) .. (2,-4);
\draw (-1,-3) ..controls +(0,0.8) and +(-0.1,0) .. ++(0.5,0.8).. controls +(0.1,0) and +(0,1.5) .. (0,-4);
\draw (1,1) ..controls +(0,-1.5)  and +(-0.1,0) .. (1.5,-0.8).. controls +(0.1,0) and +(0,-0.8) .. (2,0);
\draw (-1,1) ..controls +(0,-1.5) and +(-0.1,0) .. (-0.5,-.8).. controls +(0.1,0) and +(0,-0.8) .. (0,0);
\draw (-0.5,-0.8)-- +(2,0);
\draw (-0.5,-2.2)-- +(2,0);
%\draw (0,-4) -- +(0,4);
%\draw (2,-4) -- +(0,4);
%\draw[dashed] (1,-3) -- +(0,3);
%\draw (1,0)-- +(0,1);
%\draw (-1,-3) -- +(0,4);
\end{scope} 
\end{tikzpicture}

\vspace{0.7cm}
\begin{tikzpicture}[scale=0.5]
\input{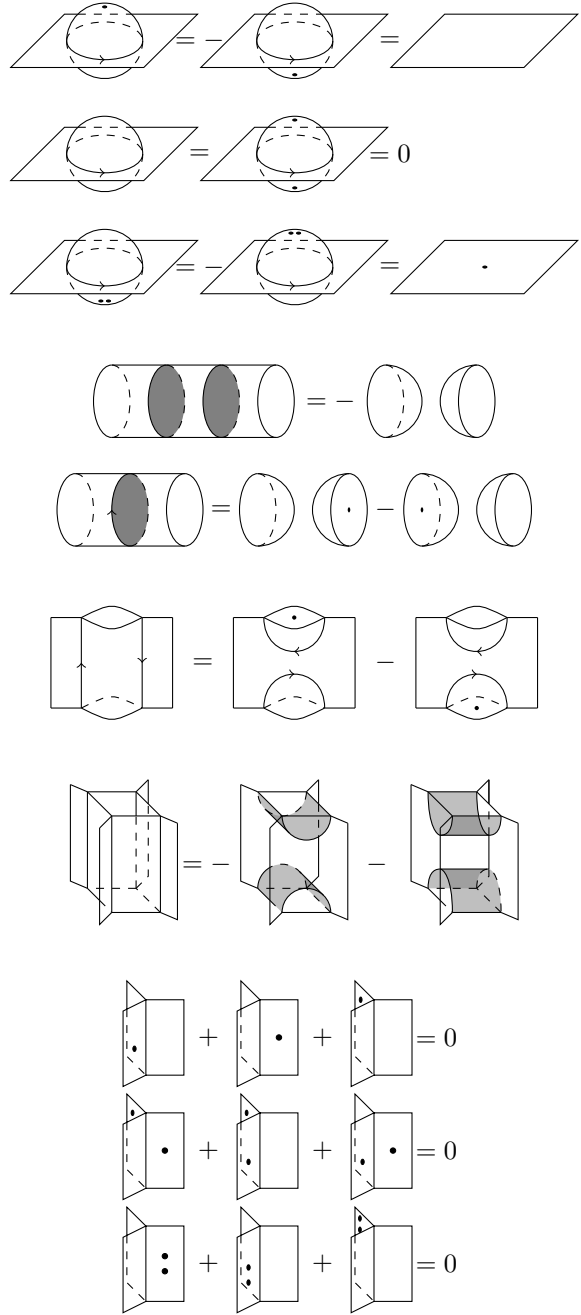} 
\end{tikzpicture}
  \caption{Local relations for 2-morphism in $\mathbb{WT}$. The first 3 lines are called bubbles relations, the 2 next are called bamboo relations, the one after digon relation, then we have the square relation and the 3 last ones are the dots migration relations.}
  \label{fig:localrel}
\end{figure}

This method allows us to define a new graded 2-category $\mathbb{WT}$. Its objects and its 1-morphisms are the ones of the 2-category $\mathcal{WT}$ while its 2-morphisms-spaces are the ones of $\mathcal{WT}$ mod out like in the last paragraph. One should notice that a $(w_b,w_t)$-foam can always be deformed into a $(\mathrm{tr}(\bar{w_b}w_t),\emptyset)$-foam and vice-versa. Khovanov's results restated in this language give that if $w_b$ and $w_t$ are $(\epsilon_0,\epsilon_1)$-web-tangles, the graded dimension of $\hom_{\mathbb{WT}}(w_t,w_b)$ is given by $\kup{\mathrm{tr}(\bar{w_b}w_t)}\cdot q^{l(\epsilon_0)+l(\epsilon_1)}$. Note that when $\epsilon_1=\emptyset$, there is no need to take the closure, because $w_b\bar{w_t}$ is already a closed web. The shift by $l(\epsilon_0)+l(\epsilon_1)$ comes from the fact that $\chi(\mathrm{tr}(\bar{w_b}w_t)) = \chi(w_t)+\chi(w_b) - (l(\epsilon_0)+l(\epsilon_1))$.
\subsection{The Kuperberg-Khovanov algebra $K^\epsilon$}
\label{sec:algebra-k_s}
We want to extend the Khovanov TQFT to the 0-dimensional objects \ie to build a 2-functor from the 2-category $\mathcal{WT}$ to the 2-category of algebras. We follow the methodology of \cite{MR1928174} and we start by defining the image of the $0$-objects~: they are the algebras $K^\epsilon$. This can be compared with \cite{2012arXiv1206.2118M}.
\begin{dfn}
  Let $\epsilon$ be an admissible finite sequence of signs. We define $\tilde{K}^\epsilon$ to be the full sub-category of $\hom_{\mathbb{WT}}(\emptyset,\epsilon)$ whose objects are non-elliptic $\epsilon$-webs. This is a graded $\QQ$-algebroid. We recall that a $k$-algebroid is a $k$-linear category. This can be seen as an algebra by setting~:
\[K^\epsilon = \bigoplus_{(w_b,w_t)\in (\mathrm{ob}(\tilde{K}^\epsilon))^2} \hom_{\mathbb{WT}}(w_b,w_t)\]
and the multiplication on $K^\epsilon$ is given by the composition of morphisms in $\tilde{K}_\epsilon$ whenever it's possible and by zero when it's not. We will denote $\tensor[_{w_1}]{K}{^{\epsilon}_{w_2}}\eqdef \hom_{\mathbb{WT}}(w_2,w_1)$. This is a unitary algebra because of proposition \ref{prop:NEFinite}. The unite element is $\sum_{w\in \mathrm{ob}(\tilde{K}_\epsilon)} 1_w$. Suppose $\epsilon$ is fixed, for $w$ a non-elliptic $\epsilon$-web, we define $P_w$ to be the left $K^\epsilon$-module~:
\[
P_w=\bigoplus_{w'\in\mathrm{ob}(\tilde{K}_\epsilon)}\hom_{\mathbb{WT}}(w,w') = \bigoplus_{w'\in\mathrm{ob}(\tilde{K}_\epsilon)} \tensor[_{w'}]{K}{^\epsilon_{w}}.
\]The structure of module is given by composition on the left.
\end{dfn}
For a given $\epsilon$, the modules $P_w$ are all projective and we have the following decomposition in the category of left $K^\epsilon$-modules~:
\[
K^\epsilon=\bigoplus_{w\in \mathrm{ob}(\tilde{K}_\epsilon)} P_w.
\]
\begin{prop}
Let $\epsilon$ be an admissible sequence of signs, and $w_1$ and $w_2$ two non-elliptic $\epsilon$-webs, then the graded dimension of $\hom_{K^\epsilon}(P_{w_1}, P_{w_2})$ is given by $\kup{(\bar{w_1}w_2)}\cdot q^{l(\epsilon)}$.
\end{prop}
\begin{proof}
  An element of  $\hom_{K^\epsilon}(P_{w_1}, P_{w_2})$ is completely determined by the image of $1_{w_1}$ and this element can be sent on any element of $\hom_{\mathbb{WT}}(P_{w_2}, P_{w_1})$, and $\dim_q(\hom_{\mathbb{WT}}(P_{w_1}, P_{w_2}))=\kup{(\bar{w_1}w_2)}\cdot q^{l(\epsilon)}$.
\end{proof}
In what follows we will prove that some modules are indecomposable, they all have finite dimension over $\QQ$ hence it's enough to show that their rings of endomorphisms contain no non-trivial idempotents. It appears that an idempotent must have degree zero, so we have the following lemma~:
\begin{lem}\label{lem:monic2indec}
  If $w$ is a non-elliptic $\epsilon$-web such that $\kup{\bar{w}w}$ is monic of degree $l(\epsilon)$, then the graded $K^\epsilon$-module $P_w$ is indecomposable.
\end{lem}
\begin{proof}
  This follows from previous discussion~: if $\hom_{K^\epsilon}(P_w,P_w)$ contained a-non trivial idempotent, there would be at least two linearly independent elements of degree 0, but $\dim((\hom_{\mathbb{WT}}(P_{w}, P_{w})_0) = a_{-l(\epsilon)}$ if we write $\kup{\bar{w}w}=\sum_{i\in \ZZ}a_iq^i$ but as $\kup{\bar{w}w}$ is symmetric (in $q$ and $q^{-1}$) of degree $l(\epsilon)$ and monic, $a_{-l(\epsilon)}$ is equal to 1 and this is a contradiction.
\end{proof}
We have a similar lemma to prove that two modules are not isomorphic.
\begin{lem}\label{lem:0monic2noiso}
  If $w_1$ and $w_2$ are two non-elliptic $\epsilon$-webs such that $\kup{\bar{w_1}w_2}$ has degree strictly smaller than $l(\epsilon)$, then the graded $K^\epsilon$-modules $P_{w_1}$  and $P_{w_2}$ are not isomorphic.
\end{lem}
\begin{proof}
  If they were isomorphic, there would exist two morphisms $f$ and $g$ such that $f\circ g=1_{P_{w_1}}$ and therefore $f\circ g$ would have degree zero. The hypothesis made implies that $f$ and $g$ (because $\kup{\bar{w_1}w_2} = \kup{\bar{w_2}w_1}$) have positive degree so that the degree of their composition is as well positive.
\end{proof}
\begin{req}
  The way we constructed the algebra $K^\epsilon$ is very similar to the construction of $H^n$ in \cite{MR1928174}. Using the same method we can finish the construction of a $0+1+1$ TQFT~:
  \begin{itemize}
  \item If $\epsilon$ is an admissible sequence of signs, then $\F(\epsilon) = K^\epsilon$.
  \item If $w$ is a $(\epsilon_1,\epsilon_2)$-web-tangle with $\epsilon_1$ and $\epsilon_2$ admissible, then 
\[\F(w) = \bigoplus_{\substack{u\in \mathrm{ob}(\tilde{K}^{\epsilon_1}) \\ v\in \mathrm{ob}(\tilde{K}^{\epsilon_2})}} \F(\bar{u}wv),
\] and it has a structure of graded $K^{\epsilon_1}$-module-$K^{\epsilon_2}$. Note that if $w$ is a non-elliptic $\epsilon$-web, then $\F(w)=P_w$.
\item If $w$ and $w'$ are two $(\epsilon_1,\epsilon_2)$-web-tangles, and $f$ is a $(w,w')$-foam, then we set 
\[
\F(f) = \sum_{\substack{u\in \mathrm{ob}(\tilde{K}^{\epsilon_1}) \\ v\in \mathrm{ob}(\tilde{K}^{\epsilon_2})}} \F(\tensor[_{\bar{u}}]{f}{_v}),
\] where $\tensor[_{\bar{u}}]{f}{_v}$ is the foam $f$ with glued on its sides $\bar{u}\times[0,1]$ and $v\times [0,1]$. This is a map of  graded $K^{\epsilon_1}$-modules-$K^{\epsilon_2}$.
  \end{itemize}
We encourage the reader to have a look at this beautiful construction for the $sl_2$ case in \cite{MR1928174}. 
\end{req}
\subsection{A decomposable module}
\label{sec:decomposable-module}
We mention that modules over $K^\epsilon$ are studied independently in \cite{2012arXiv1206.2118M} where they compute the split Grothendieck group of $K^{\epsilon}$. However the problem to give an explicit basis of indecomposable projective module remains open. 
As we discussed before, to prove that a module is decomposable, it's enough to show that its ring of endomorphisms contains a non-trivial idempotent. In this subsection we show that a certain module $P_w$ is decomposable. This is actually already known (see for example \cite{MR2457839}), but we give here details of the calculus.

In what follows we set $\epsilon$ to be the sequence $(+,-,-,+,+,-,-,+,+,-,-,+)$ (so that $l(\epsilon)=12$) and $w$ and $w_0$ the $\epsilon$-webs given on figure~\ref{fig:thewebw}.
\begin{figure}[h]
  \centering
\begin{tikzpicture}[scale= 0.5]
  \input{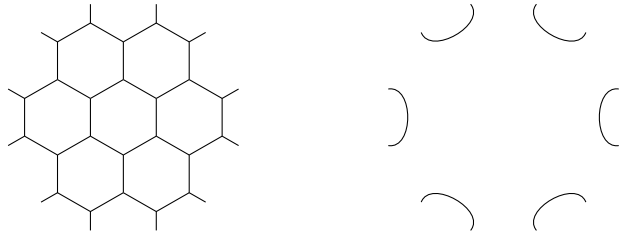}
\end{tikzpicture}
  \caption{The $\epsilon$-webs $w$ (on the left) and $w_0$ (on the right), to fit in formal context of the 2-category one should stretch the outside edges to horizontal line below the whole picture, we draw it this way to enjoy more symmetry. To simplify we didn't draw the arrows.}
  \label{fig:thewebw}
\end{figure}
We will as well need some specific foams we describe them via movies on figure~\ref{fig:fgfoams} (the elementary movies are birth, death and saddle as for classical surfaces and we add zip and unzip, see \cite{MR2100691} for details).
\begin{figure}[h]
  \centering
  \begin{tikzpicture}[scale =0.8]
    %  \begin{scope}[scale= 0.1]
%    \foreach \n/\a in {a/-210,b/-150,c/-90,d/-30,e/30,f/90}
% {
% \draw (\a:1) -- (\a:2);
% \draw (\a:1) -- (\a+60:1);
% \draw (\a:2) -- ++(\a-60:1) -- ++(\a-120:1) -- ++(\a-180:1);
% \draw (\a:2) ++(\a-60:1) .. controls +(\a:\a/120*\a/30+\a/60+1) and +(0,0) .. (\a/30+2,-4);
% \draw (\a:2) ++(\a-60:1) ++(\a-120:1) -- ++(\a-60:0.5);
% }
% \end{scope}
\begin{scope}[xshift=-8cm, yshift= -2cm]
\draw[draw=  gray, line width = 2mm] (-2,0) --  (10,0); 
\draw [draw = gray, line width = 2mm] (10,4) -- (-2,4); 
\draw[draw=  gray, very thick] (-2,0) -- (-2,4);
\draw[draw=  gray, very thick] (2,0) -- (2,4);
\draw[draw=  gray, very thick] (6,0) -- (6,4);
\draw[draw=  gray, very thick] (10,0) -- (10,4);
%\draw[draw=  gray, very thick] (-4,0) -- (-4,3);
%\draw[draw=  gray, very thick] (16,0) -- (16,3);
\draw[draw= white, dotted, line width =1.2mm] (-2,0) -- (10,0);
\draw[draw= white, dotted, line width = 1.2mm] (-2,4) -- (10,4);
\end{scope}
 \begin{scope}[scale= 0.5]
   \foreach \n/\a in {a/30,b/90,c/150,d/210,e/270,f/330}
{
\draw (\a:1) -- (\a:2);
\draw (\a:1) -- (\a+60:1);
\draw (\a:2) -- ++(\a-60:1) -- ++(\a-120:1) -- ++(\a-180:1);
\draw (\a:2) ++(\a-60:1) -- ++(\a:0.5);
\draw (\a:2) ++(\a-60:1) ++(\a-120:1) -- ++(\a-60:0.5);
}
\end{scope}

\begin{scope}[xshift =-4cm, scale=0.5]
   \foreach \n/\a in {a/30,b/90,c/150,d/210,e/270,f/330}
{
\draw (\a+30:1.5) circle (0.5);
\draw (\a:2.5)++(\a+60:1).. controls (\a+10:2.5) and (\a+50:2.5) .. +(\a+120:1.5);
%\draw (\a:1) -- (\a:2);
%\draw (\a:1) -- (\a+60:1);
%\draw (\a:2) -- ++(\a-60:1) -- ++(\a-120:1) -- ++(\a-180:1);
%\draw (\a:2) ++(\a-60:1) ++(\a:0.5).. controls +(\a:);
%\draw (\a:2) ++(\a-60:1) ++(\a-120:1) -- ++(\a-60:0.5);
}
\end{scope}

\begin{scope}[xshift =-8cm, scale =0.5]
   \foreach \n/\a in {a/30,b/90,c/150,d/210,e/270,f/330}
{
\draw (\a:2.5)++(\a+60:1).. controls (\a+10:2.5) and (\a+50:2.5) .. +(\a+120:1.5);
%\draw (\a:1) -- (\a:2);
%\draw (\a:1) -- (\a+60:1);
%\draw (\a:2) -- ++(\a-60:1) -- ++(\a-120:1) -- ++(\a-180:1);
%\draw (\a:2) ++(\a-60:1) ++(\a:0.5).. controls +(\a:);
%\draw (\a:2) ++(\a-60:1) ++(\a-120:1) -- ++(\a-60:0.5);
}
\end{scope}
  \end{tikzpicture}
  \caption{The $(w_0,w)$-foam $f$ (it means that $f$ belongs to $\hom_{\mathbb{WT}}(w,w_0)$) is described by the movie from left to right and the $(w,w_0)$-foam $g$ is described by the movie from right to left (we have $g=\bar{f}$). We specify for $f$~: at the first step we perform 6 births and then, at the second step we zip 12 times, this leads to morphisms of degree 0.}
  \label{fig:fgfoams}
\end{figure}
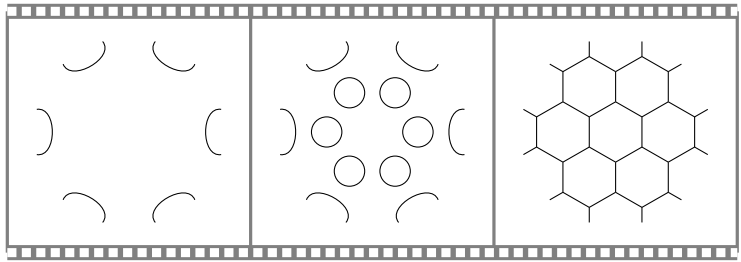
\begin{prop}
  The $K^\epsilon$-module $P_w$ is decomposable.
\end{prop}
First we should quote that this web was the counter-example pointed out by Kuperberg and Khovanov \cite{MR1684195} to prove that the web basis is not dual canonical. To show this proposition we just have to find a non-trivial idempotent. 
\begin{proof}
We claim that $e\eqdef\frac12 g\circ f$ is a $(w,w)$-foam whose associated endomorphism is an idempotent different from $0$ and from $1_{P_w}$. To prove this we first show that $f\circ g =2\cdot 1_{w_0}$. First notice that $f\circ g$ belongs to $\hom_{\mathbb{WT}}(w_0,w_0)$ and has degree 0. This space has a graded dimension given by $q^{12}\cdot\kup{\bar{w_0}w_0}= q^{12}\cdot [3]^6$, and this shows that the space of degree 0 endomorphism of $P_{w_0}$ is 1-dimensional and hence $f\circ g$ is a multiple of $1_{w_0}$.  Seen as a foam and forgetting the decomposition of its boundary we can consider $g\circ f$ as $(\emptyset,\bar{w_0}w_0)$-foam and in this context $f\circ g$ is a multiple of $h$, the $(\emptyset,\bar{w_0}w_0)$-foam given by 6 cups. To evaluate the scalar multiple between them, we complete these $(\emptyset,\bar{w_0}w_0)$-foams to obtain closed foams by gluing $j$, the $(\bar{w_0}w_0,\emptyset)$-foam which consists of 6 caps with two dots on each cap. The closed foam $h \circ j$ is a closed foam which consists of 6 spheres with 2 dots on each, hence it's evaluation $\mathcal{F}(h\circ j)$ is equal to $(-1)^6=1$. Now let us evaluate $\mathcal{F}(f\circ g\circ j)$ by using the bubble relations next to the cups of $j$, we have $\mathcal{F}(f\circ g \circ j)=(-1)^3\mathcal{F}(t)$, where $t$ is a torus with 6 disks inside and one dot per section of the torus (see figure~\ref{fig:foamt}). To evaluate $\mathcal{F}(t)$ one can perform surgeries on all portions, this gives us a priori $3^6$ terms with plus signs, all are disjoint union of 6 dotted theta foams just 2 of this terms are non-zero~: and for the two of them 3 theta foams evaluate on $-1$ and $3$ theta foams evaluate on $+1$, so that $\mathcal{F}(f\circ g\circ j) = -\mathcal{F}(t) = 2$. 
\begin{figure}[h!]
  \centering
  \begin{tikzpicture}[scale=0.5]
     \begin{scope}
   [yscale = {1}, xscale={1},decoration={markings, mark=at
     position 0.5 with {\arrow{>}}},postaction={decorate}]
\fill[color = gray] (0,1) ellipse (0.5 and 1);
\fill[color = gray] (2,1) ellipse (0.5 and 1);
\fill[color = gray] (4,1) ellipse (0.5 and 1);
\fill[color = gray] (6,1) ellipse (0.5 and 1);
\fill[color = gray] (8,1) ellipse (0.5 and 1);
\fill[color = gray] (10,1) ellipse (0.5 and 1);
   \draw (0,0) -- (10,0);
   \draw (0,2) -- (10,2);
   \draw[postaction= {decorate}] (0,0) arc (270:90:0.5 and 1);
   \draw[postaction= {decorate}] (4,0) arc (270:90:0.5 and 1);
   \draw[postaction= {decorate}] (8,0) arc (270:90:0.5 and 1);
   \draw[postaction= {decorate}] (2,2) arc (90:270:0.5 and 1);
   \draw[postaction= {decorate}] (6,2) arc (90:270:0.5 and 1);
   \draw[postaction= {decorate}] (10,2) arc (90:270:0.5 and 1);
   \draw[dotted] (0,0) arc (-90:90:0.5 and 1);
\draw[dotted] (2,0) arc (-90:90:0.5 and 1);
\draw[dotted] (4,0) arc (-90:90:0.5 and 1);
\draw[dotted] (6,0) arc (-90:90:0.5 and 1);
\draw[dotted] (8,0) arc (-90:90:0.5 and 1);
\draw[dotted] (10,0) arc (-90:90:0.5 and 1);
\filldraw (1,1) circle (2 pt);
\filldraw (3,1) circle (2 pt);
\filldraw (5,1) circle (2 pt);
\filldraw (7,1) circle (2 pt);
\filldraw (9,1) circle (2 pt);
\filldraw (-1,1) circle (2 pt);
\draw (0,0) .. controls +(-3, 0) and +(0,-1) .. (-3,2).. controls +(0,1)  and +(-5,0).. (5, 6) .. controls +(5,0) and +(0,1) .. (13,2) .. controls +(0,-1) and +(3,0) .. (10,0);
\draw (0,2) .. controls +(-2, 0)  and +(-5,0).. (5, 4) .. controls +(5,0) and +(2,0).. (10,2);
      % \draw (4,0) ellipse (0.5 and 1);
      % \draw[dotted] (0,-1) arc (-90:90:0.5 and 1);
      % \draw (0,1) arc (90:270:0.5 and 1);
      % % \draw[fill=gray, draw=black] (2,0) ellipse (0.5 and 1);
      % \draw (1.3,-2) -- (1.3,1.5);
      % \draw (1.3,-2) -- (2.7,-1.5);
      % \draw (1.3,1.5) -- (2.7,2);
      % \draw (2.7,-1.5) -- (2.7,-1);
      % \draw (2.7,1) -- (2.7,2);
      % \draw[dotted] (2.7,1)--(2.7,-1);
      % \draw (0,1) -- (1.3,1);
      % \draw (0,-1) -- (1.3,-1);
      % \draw (2,1) -- (4,1);
      % \draw (2,-1) -- (4,-1);
      % %\draw[dotted] (2,1) -- (1.2,1);
      % %\draw[dotted] (2,-1) -- (1.2,-1);
      % %\node at (2,0) {$S$};
     % \node at (2,-2) {$f$};
\end{scope}
  \end{tikzpicture}
  \caption{The closed foam $t$.}
  \label{fig:foamt}
\end{figure}
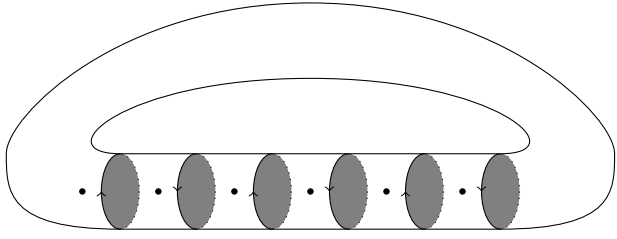
  This gives us that $g\circ f =2\cdot 1_{w_0}$. It is then very easy to check that $e$ is an idempotent~: $e\circ e = \frac14 f\circ g \circ f \circ g = \frac12 f\circ 1_{w_0}\circ g =e$, and it is as well straightforward to check that $e$ is not equal to zero~: $g\circ e\circ f = 2\cdot 1_{w_0}\neq 0$. We now need to show that $e$ is not equal to $1_w$. If it were $f$ and $\frac12g$ would give isomorphism between $P_w$ and $P_{w_0}$. But their spaces of endomorphism don't have the same graded dimensions\footnote{In fact we have $\kup{w\bar{w}}=2q^{-12} + 80q^{-10} + 902q^{-8} + 4604q^{-6} + 13158q^{-4} + 23684q^{-2} + 28612 + 23684q^2 + 13158q^4 + 4604q^6 + 902q^8 + 80q^{10} + 2q^{12}$ and $\kup{w_0\bar{w_0}} =q^{-12} + 6q^{-10} + 21q^{-8} + 50q^{-6} + 90q^{-4} + 126q^{-2} + 141 + 126q^2 + 90q^4 + 50q^6 + 21q^8 + 6q^{10} + q^{12} $. We used Lukas Lewark's program \cite{WD2011Lewark} to compute $\kup{w\bar{w}}$.} so they cannot be isomorphic. This shows that $P_w$ is decomposable and that $P_{w_0}$ is a $K^\epsilon$-sub-module of $P_w$. 
\end{proof}

\section{An indecomposability result}
The aim of this section is to give a rather large family of $\epsilon$-webs whose associated $K^\epsilon$-modules are indecomposable, and pairwise non-isomorphic.
\subsection{Superficial webs, semi-non elliptic webs}
\begin{dfn}
 If $\epsilon$ is an admissible sequence of signs, we denote by $S^\epsilon$ the $\QQ[q,q^{-1}]$-module generated by isotopy classes of $\epsilon$-webs and subjected to the Kuperberg relations (see proposition \ref{prop:Kup}).
\end{dfn}
\begin{prop}[Kuperberg]
  The $\QQ[q,q^{-1}]$-module $S^\epsilon$ is free and freely generated by the non-elliplitic $\epsilon$-webs.
\end{prop}
Let us consider a $\epsilon$-web, there are finitely many connected components of $\RR^2\setminus w$ (we call them \emph{faces} even if some may not be homeomorphic to a disk). As $w$ is compact, there is just one of these faces which is unbounded. We call it \emph{the unbounded face}. Note that because of the geometric requirements on $\epsilon$-webs, all the points of $\epsilon$ are in the adherence of the unbounded face. We say that two faces are \emph{adjacent} if an edge of $w$ is included in the intersection of their adherences.
\begin{dfn}\label{dfn:nested}
  A face of an $\epsilon$-web is said to be \emph{nested} if it is not adjacent to the unbounded face. An $\epsilon$-web with no nested face is called \emph{superficial}.
\end{dfn}
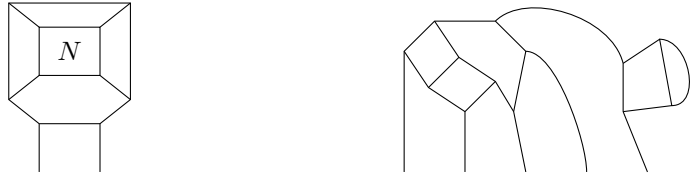
\begin{figure}[h!]
  \centering
  \begin{tikzpicture}[scale=0.8]
    \begin{scope}
   [yscale = {0.8}, xscale={1},decoration={markings, mark=at
     position 0.5 with {\arrow{>}}},postaction={decorate}]
% \coordinate [label=below:{$-$}] (A) at (+0.16,0);
% \coordinate [label=below:{$-$}] (B) at (1.5,0);
% \coordinate [label=below:{$-$}] (C) at (2.83,0);
% \coordinate [label=above:{$-$}] (D) at (0,4);
% \coordinate [label=above:{$-$}] (E) at (1,4);
% \coordinate [label=above:{$+$}] (F) at (2,4);
% \coordinate [label=above:{$+$}] (G) at (3,4);
\draw (0,0) -- (0,1) -- (-0.5,1.5) -- (-0.5,3.5) -- (1.5,3.5) -- (1.5,1.5) -- (1,1) -- (1,0);
\draw (0,2) -- (0, 3) -- (1,3) -- (1,2) --cycle ;
\draw (0,1) -- (1,1);
\draw (-0.5,1.5) -- (0,2);
\draw (-0.5,3.5) -- (0,3);
\draw (1.5,3.5) -- (1,3);
\draw (1.5,1.5) -- (1,2);
\node at (0.5, 2.5) {$N$};
\end{scope}

\begin{scope}[xshift = 6cm, yscale = {1}, xscale={1},decoration={markings, mark=at
     position 0.5 with {\arrow{>}}},postaction={decorate}]
\draw (0,0) -- (0,2) -- (0.5,2.5) -- (1.5, 2.5) -- (2,2) -- (1.8,1) -- (2,0);
\draw (1,0) --(1,1)-- (0.4,1.4) -- (0,2);
\draw (0.5,2.5) -- (0.9,1.9) -- (1.5, 1.5) -- (1.8,1);
\draw (0.4,1.4) -- (0.9,1.9);
\draw (1,1) -- (1.5, 1.5);
\draw (2,2) .. controls +(0.5,0) and +(0,0.5) .. (3,0);
\draw (1.5, 2.5) .. controls +(0.5,0.5) and +(-0.2,0.8) .. (3.6,1.8) -- (3.6,1) -- (4,0);
\draw (3.6,1.8) -- (4.2,2.2) -- (4.4, 1.1) -- (3.6,1);
\draw (4.2,2.2) .. controls +(0.5,0) and + (0.5,0) .. (4.4,1.1);  
\end{scope}
  \end{tikzpicture}
  \caption{On the left an $\epsilon$-web with a nested face (marked by a $N$), on the right a superficial (and elliptic) $\epsilon$-web with two blocks.}
  \label{fig:ex_nested}
\end{figure}
The aim of this section is to prove the next theorem~:
\begin{thm}\label{thm:superficial2monic}
  Let $\epsilon$ be an admissible sequence of signs, let $w$ be a superficial non-elliptic $\epsilon$-web, then the $K^{\epsilon}$-module $P_w$, is indecomposable. Furthermore, if $w'$ is another superficial non-elliptic $\epsilon$-web different from $w$, then $P_w$ and $P_{w'}$ are not isomorphic as $K^\epsilon$-modules.
\end{thm}
We begin by a few technical definitions. In an $\epsilon$-web, let's consider all the faces but the unbounded one. They come in adjacency classes. We call such an adjacency a \emph{block}. In other words, blocks are connected components of the graph obtained from the dual graph by removing the vertex corresponding to the unbounded component and all the edges involving this vertex.  
\begin{dfn}  An $\epsilon$-web is semi-non-elliptic if it contains no verticeless loop, no
  digon, and at most one square per block.
\end{dfn}
\begin{lem}\label{lem:sneni}
  If an $\epsilon$-web $w$ superficial and semi-non-elliptic then in the skein module $S^\epsilon$ it is equal to a sum of superficial non-elliptic $\epsilon$-webs with less vertices.
\end{lem}
Here, by  ``sum'' here we mean linear combination with only positive integer coefficients.
\begin{proof}
  We prove this by induction on the number of trivalent vertices. If $w$ is
  already non-elliptic then there is nothing to prove. Else there is
  at least a square somewhere. Then if 
\[
w = \foamsquarecircle[0.4],
\]
we have \[ 
w = \vcenter{\hbox{
    \begin{tikzpicture}[scale=0.3,decoration={markings, mark=at
        position 0.5 with {\arrow{>}}},postaction={decorate}]
      \draw[dotted] (1.5,1.5) circle (1.414cm);
      \draw[postaction= {decorate}] (0.5,0.5) .. controls (1,1) and (1,2) .. (0.5, 2.5);
      \draw[postaction= {decorate}] (2.5,2.5) .. controls (2,2) and (2,1) .. (2.5, 0.5);
    \end{tikzpicture}}}+
\vcenter{\hbox{
 \begin{tikzpicture}[scale=0.3,decoration={markings, mark=at
        position 0.5 with {\arrow{>}}},postaction={decorate}]
      \draw[dotted] (1.5,1.5) circle (1.414cm);
      \draw[postaction= {decorate}] (0.5,0.5) .. controls (1,1) and (2,1) .. (2.5, 0.5);
      \draw[postaction= {decorate}] (2.5,2.5) .. controls (2,2) and (1,2) .. (0.5, 2.5);
    \end{tikzpicture}}}.
\]
Let $w_1$ and $w_2$ be these two webs. As $w$ is superficial, one of
the 4 faces around the square should be the unbounded face $U$, we can
suppose it's the one on the top. We'll now inspect the faces
of $w_1$ and $w_2$, see figure \ref{fig:w1w2} for names of faces.

\begin{figure}[h!]
  \centering
 \begin{tikzpicture}[scale=1]
\input{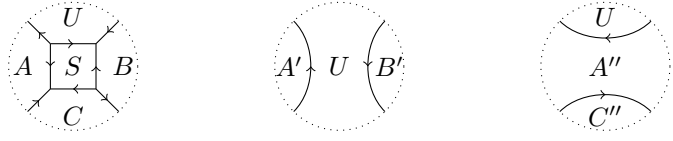}
\end{tikzpicture}
  \caption{The webs $w$, $w_1$ and $w_2$.}
  \label{fig:w1w2}
\end{figure}
The faces $A'$ and $B'$ of $w_1$ may be two squares but as $w$ is
superficial, $A'$ and $B'$ are on different blocks, else there
would exist a path of faces in $w$ from $A$ to $B$ disjoint of $C$ and
$S$ and $C$ would be nested. As the square $S$ is the only square of
its block in $w$, the block of $A'$ and the block of $B'$ have at most
one square, hence $w_1$ is superficial and semi-non-elliptic. For $w_2$
now, the only possibly new square is the face $C$ and if it is a
square, it's the only one in its block, hence $w_2$ is superficial and 
semi-non-elliptic. To conclude, just notice that $w_1$ and $w_2$ have
less vertices than $w$. 
\end{proof}
\begin{dfn}
An $\epsilon$-web $w$ is called $1$-elliptic if it contains no circle, no digon and if there are at most one square in each block except in one where it can contains at most two.
\end{dfn}
\begin{lem}\label{lem-1-elliptic}
If $w$ is a superficial $1$-elliptic $\epsilon$-web then there exist some non-elliptic $\epsilon$-webs $w_i$ and some symmetric polynomials $P_i$ in $\mathbb{N}[q,q\m]$ with degree at most 1 such that in $S^\epsilon$ we have $w = \sum P_i w_i$.
\end{lem}
\begin{proof}
  We prove this result by recursion on the number of vertices. If there are no vertices the web is non-elliptic and this is done.  If $w$ is semi-non-elliptic then the result comes from lemma \ref{lem:sneni}. So it remains to understand the case where there is one block with  two squares. If the two squares are far from each other (we mean that they don't share any edge) then we can proceed as in the proof of lemma \ref{lem:sneni} and prove that $w$ is a sum (with positive integer coefficient) of 1-elliptic webs and then we conclude by recursion. Now we study the case where the two squares touch each other. As $w$ is superficial the two squares $S_1$ and $S_2$ (see figure \ref{fig:2-square} for the notations) should touch the unbounded face. Furthermore $w$ is equal to $[2]w_1+w_2$ (see figure \ref{fig:w1w2-2}).
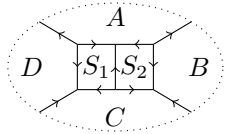
\begin{figure}[h!]
  \centering 
   \begin{tikzpicture}[scale=1]
\begin{scope}[yscale ={0.6}, xscale={1},decoration={markings, mark=at
     position 0.5 with {\arrow{>}}},postaction={decorate}]
      \draw[dotted] (1.5,1.5) circle (1.414 cm);
      %\draw[postaction = {decorate}] (0,0)--(3,0);
      \draw[postaction = {decorate}] (0.5,0.5)--(1,1);
      %\draw[postaction = {decorate}] (0,0)--(0,3);
      %\draw[postaction = {decorate}] (3,3)--(0,3);
      \draw[postaction = {decorate}] (2,2) -- (2.5,2.5);
      \draw[postaction = {decorate}] (1.5,1)--(1.5,2);
      \draw[postaction = {decorate}] (1.5,1)--(1,1);
      \draw[postaction = {decorate}] (1.5,1)--(2,1);
      \draw[postaction = {decorate}] (2,2)--(2,1);
      \draw[postaction = {decorate}] (2.5,0.5) -- (2,1);
      \draw[postaction = {decorate}] (2,2) -- (1.5,2);
      \draw[postaction = {decorate}] (1,2) -- (1.5,2);
      \draw[postaction = {decorate}] (1,2)--(1,1);
      \draw[postaction = {decorate}] (1,2)--(0.5,2.5);
      \node at (1.5,0.4) {$C$};
      \node at (1.25,1.5) {$S_1$};
      \node at (1.75,1.5) {$S_2$};
      \node at (1.5,2.6) {$A$};
      \node at (0.4,1.5) {$D$};
%      \node at (1.5,1.5) {$C$};
      \node at (2.6,1.5) {$B$};
 \end{scope}
 
 \end{tikzpicture}
  \caption{Situation where two squares touch each other.}
  \label{fig:2-square}
\end{figure}
Now let us consider the different cases. There are two different situations, either $A$ or $C$ is unbounded (the two situations are symmetric) or both $B$ and $D$ are unbounded. If $A$ is unbounded, then $D'$ and $B'$ may be squares but on different blocks, otherwise, $C$ would be nested, and hence $w_1$ is superficial semi-non-elliptic. On the other hand $B''$ and $D''$ have at least 6 vertices, and $C''$ may be a square but in every case the $\epsilon$-web $w_2$ is superficial semi-non-elliptic. It remains the case where $B$ and $D$ are unbounded. The face $A'$ has at least 6 vertices so that $w_1$ is semi-non-elliptic. On the other hand $A''$ and $C''$ may be squares but there are clearly on different blocks so that $w_2$ is semi-non-elliptic. Hence using lemma \ref{lem:sneni} we conclude.
\begin{figure}[h!]
\centering
%\hspace{1cm}
\begin{tikzpicture}[scale=1]
  \input{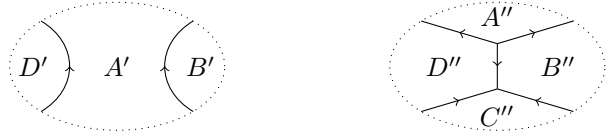}
\end{tikzpicture}
  \caption{The webs $w_1$ and $w_2$.}
  \label{fig:w1w2-2}
\end{figure}
\end{proof}
\begin{dfn}\label{dfn:semi-super}
  An $\epsilon$-web is said to be \emph{semi-superficial} if it contains no circle, no digon, only one square, only one nested face and if this nested face is an hexagon which share a side with the square.
\end{dfn}
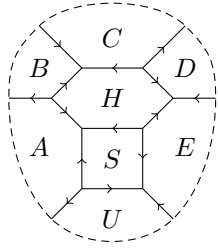
\begin{figure}[h!]
  \centering
  \begin{tikzpicture}[scale=0.8]
     \begin{scope}[scale=1,decoration={markings, mark=at
        position 0.5 with {\arrow{>}}},postaction={decorate}]
\coordinate (A) at (-0.5, -1);
\coordinate (B) at (-0.5,  0);
\coordinate (C) at (-1,  0.5);
\coordinate (D) at (-0.5,  1);
\coordinate (A1) at (0.5, -1);
\coordinate (B1) at (0.5,  0);
\coordinate (C1) at (1,  0.5);
\coordinate (D1) at (0.5,  1);
\coordinate (e1) at (-1,  -1.5);
\coordinate (e2) at (-1.7,  0.5);
\coordinate (e3) at (-1.2, 1.7);
\coordinate (e6) at (1,  -1.5);
\coordinate (e5) at (1.7,  0.5);
\coordinate (e4) at (1.2, 1.7);
\draw[postaction= {decorate}] (A) -- (B);
\draw[postaction= {decorate}] (A) -- (A1);
\draw[postaction= {decorate}] (A) -- (e1);
\draw[postaction= {decorate}] (e6) -- (A1);
\draw[postaction= {decorate}] (B1)--(A1);
\draw[postaction= {decorate}] (B1)--(B);
\draw[postaction= {decorate}] (B1)--(C1) ;
\draw[postaction= {decorate}] (C)--(B);
\draw[postaction= {decorate}] (C) --(D);
\draw[postaction= {decorate}] (C)--(e2);
\draw[postaction= {decorate}] (e5)-- (C1);
\draw[postaction= {decorate}] (D1)-- (D);
\draw[postaction= {decorate}] (D1) --(C1);
\draw[postaction= {decorate}] (D1) -- (e4);
\draw[postaction= {decorate}] (e3)  --(D);
\draw[densely dashed] (e1) .. controls +(-0.5,+0.5) and +(0,-0.7) .. (e2) .. controls +(0,0.6) and +(-0.4,-0.4) .. (e3).. controls +(0.5,0.5) and +(-0.5,+0.5) .. (e4) .. controls +(0.4,-0.4) and +(0,0.6) ..(e5).. controls +(0,-0.7) and +(0.5,0.5) ..(e6).. controls +(-0.5,-0.5) and +(0.5,-0.5)..(e1);
\node at (0,-1.5) {$U$};
\node at (0,-0.5) {$S$};
\node at (0,0.5) {$H$};
\node at (0,1.5) {$C$};
\node at (-1.2, -0.3) {$A$};
\node at (1.2, -0.3) {$E$};
\node at (-1.2, 1) {$B$};
\node at (1.2, 1) {$D$};
\end{scope}
  \end{tikzpicture}
  \caption{Semi-superficial web. The faces $A$, $B$, $C$, $D$, $E$, $H$ and $S$ are bounded. The label $U$ shows the undounded face.}
  \label{fig:semi_superficial}
\end{figure}
\begin{lem}\label{lem-semi-superficial}
If $w$ is a semi-superficial $\epsilon$-web then there exist some non-elliptic $\epsilon$-webs $w_i$ and some symmetric polynomial $P_i$ in $\mathbb{N}[q,q\m]$ with degree at most 1 such that in $S^\epsilon$ we have $w = \sum P_i w_i$.
\end{lem}
\begin{proof}
We take the notation of figure~\ref{fig:semi_superficial}. We perform the square reduction on $S$ and then on $H$ in the configuration where it's possible so that in the skein module $S^\epsilon$ we have : $w=w_1+w_2+w_3$. See figure \ref{fig:semisuper-reductions} for a description of the $w_i$.
\begin{figure}[h!]
  \centering
  \begin{tikzpicture}[scale=0.8]
    \input{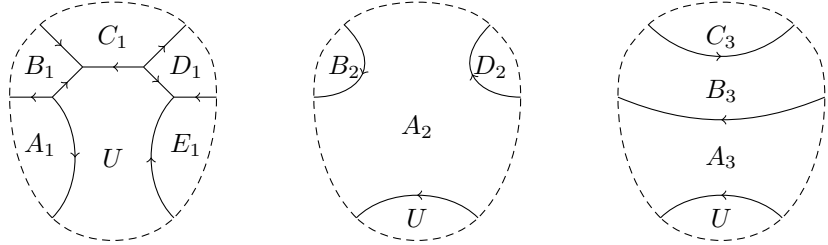}
  \end{tikzpicture}
  \caption{From left to right $w_1$, $w_2$ and $w_3$.}
  \label{fig:semisuper-reductions}
\end{figure}
\begin{enumerate}
\item \label{enum:fact1}
The web $w_1$ is superficial and $1$-elliptic. The superficiality is clear, the only two faces which may be squares are $A_1$ and $E_1$.
\item\label{enum:fact2} The web $w_2$ is superficial and $1$-elliptic. The face $A_2$ has at least 6 sides because it has the sides of $A$ minus three of them and the sides of $E$ minus three of them then it leads to at least 6 sides. The superficiality is clear, the only two faces which may be squares are $B_2$ and $D_2$.
\item\label{enum:fact3} The web $w_3$ is superficial and semi-non-elliptic. The superficiality is clear. The only face which can be a square is $C_3$.
\end{enumerate}
From (\ref{enum:fact1}), (\ref{enum:fact2}) and (\ref{enum:fact3}), and the lemmas \ref{lem-1-elliptic} and \ref{lem:sneni} we conclude easily.
\end{proof}
Thanks to lemma \ref{lem:monic2indec} and \ref{lem:0monic2noiso}, to prove the theorem \ref{thm:superficial2monic}, it's enough to prove the following proposition~:
\begin{klem}%\label{prop:superficial2monic2}
  Let $\epsilon$ be a sequence of signs of length $n$. Let $w_1$ and
  $w_2$ be $\epsilon$-webs, suppose $w_1$ and $w_2$
  are superficial and non-elliptic. If $w_1 = w_2$ then $\langle \bar{w_1}w_2\rangle$ is monic and has degree $n$, else $\deg \langle \bar{w_1}w_2\rangle < n$.
\end{klem}

 \subsection{Proof of the key lemma}
\newcommand{\w}{\ensuremath{\mathbf{w}}}
\label{sec:proof-theo-monic}
\begin{dfn}
  \label{dfn:order}
  We consider the set $W$ of pairs of (isotopy class) of superficial non-elliptic webs with the same boundary. On this set we define a partial order~: $(w_1,w_2)<(w'_1,w'_2)$ if and only if either  $l(\partial w_1) < l(\partial w'_1)$ or  $l(\partial w_1) = (\partial w'_1)$ and $\# V(w_1) + \#V(w_2) < \#V(w'_1) + \#V(w'_2)$. An element $\w =(w_1, w_2)$ of $W$ is \emph{symmetric} if $w_1=w_2$.
\end{dfn}
This order is meant to encode the complexity of a web. We could have sharpened it but this won't be necessary for our purposes. For $n\in\NN$, we denote $W_n$ the subset of $W$ in which the webs in the pairs have a boundary of length $n$. For example $W_0 =\{(\emptyset,\emptyset)\}$ and $W_1=\emptyset$. As every $W_n$ is finite, $W$ is a well quasi-ordered set with one minimal element : $(\emptyset,\emptyset)$.
\begin{dfn}
  Let $\w =(w_1,w_2)$ be an element of $W_n$, we say that $\w$ is \emph{nice} if
  \begin{itemize}
  \item the element $\w$ is symmetric, $\langle \bar{w_1}w_2\rangle$ is monic and has degree $n$,
  \item or if $\w$ is not symmetric and the degree of $\langle \bar{w_1}w_2\rangle$ is strictly smaller than $n$.
  \end{itemize}
\end{dfn}
The key lemma is rephrased in this vocabulary by the following proposition~:
\begin{prop}\label{prop:superfical2monic}
  Every element of $W$ is nice.
\end{prop}
\begin{notation}\label{dfn:border}
  If $\w =(w_1,w_2)$ is an element of $W$, then $\bar{w_1}w_2$ is a closed web, in what follows it will be practical to consider the isotopy class of $\bar{w_1}w_2$ but to keep in mind that this closed web has two different parts : the $w_1$ part and the $w_2$ part. In order to remember this, while performing an isotopy of $\bar{w_1}w_2$ we keep track of the real line where $\epsilon$ was lying (this is where $\bar{w_1}$ and $w_2$ are glued together). This curve is the \emph{border} between $w_1$ and $w_2$. This will be depicted by a dashed line. 
\end{notation}
When performing a reduction of $\bar{w_1}w_2$ (removing a circle or a reduction of a square or of a digon), we should keep the border in mind so that the reduction leads to a new closed web $w'$ which can be understood as $\bar{w'_1}w'_2$, where $w'_i$ is the same as $w_i$ except in the place where we perform the reduction. Note that when the reduction takes place next to the border we have to specify how the border behaves with respect to the reduction so that $w'_1$ and $w'_2$ are well defined. We may as well perform moves of the border, this is to be understood as that we change the pair $(w_1, w_2)$ in the way given by the changed border (this may change the boundary).
\begin{lem}\label{lem:circle2nice}
  If  $\w =(w_1,w_2)$ is an element of $W$ such that $\bar{w_1}w_2$ contains a circle $C$ then there exists $\w' =(w'_1,w'_2)$ with $\w'<\w$ such that $\w'$ nice implies $\w$ nice.
\end{lem}
\begin{proof}
  As $w_1$ and $w_2$ are non-elliptic, the circle must intersect the border. We shall consider two cases : the border cuts $C$ in two points or in at least four points\footnote{A circle and the border must intersect in an even number of points for orientation reasons.}.

First consider the case where the border cuts $C$ in two points. Then it separates the $C$ into two half-circles $C_i\subset w_i$ for $i\in\{1,2\}$. We denote $w'_i$ the $\epsilon'$-web $w_i\setminus C_i$. The sequence $\epsilon'$  is equal to $\epsilon$ with a '$+$' and a '$-$' removed hence $l(\epsilon')=l(\epsilon)-2$. We have $\kup{\bar{w_1}w_2} = [3]\kup{\bar{w_1'}w'_2}$ and $\w'$ is symmetric if and only if $\w$ is. It's clear that $\w'=(w'_1,w'_2)$ belongs to $W$ and $\w'<\w$ and that if $\w'$ is nice then $\w$ is nice.

The second case : $C$ meets the border in a least four points. Then $w_1\neq w_2$. Consider once more the pair of $\epsilon'$-webs $\w'=(w'_1,w'_2)$ obtained from $\w$ by removing the circle $C$. The length of $\epsilon'$ is at most $l(\epsilon) -4$. Then if $\w'$ is nice then $\deg \kup{\bar{w'_1}w'_2})\leq l(\epsilon)-4$, and then~:
\[
  \deg \kup{\bar{w_1}w_2}= \deg([3]\kup{\bar{w'_1}w_2'}) = 2 + \deg \kup{\bar{w'_1}w'_2} \leq l(\epsilon) -2.
\]
This is clear that $\w'$ is in $W$ and that $\w'<\w$. And so we are done.
\end{proof}
\begin{lem}\label{lem:notconnected2nice}
  If  $\w =(w_1,w_2)$ is an element of $W$ such that $\bar{w_1}w_2$ is not connected then there exist $\w'$  and $\w''$ with $\w'<\w$ and $\w''<\w$ such that if $\w'$ and $\w''$ are nice then $\w$ is nice.
\end{lem}
\begin{proof}
  Consider one connected component of $\bar{w_1}w_2$ and denote it by $u$ and denote $v$ the complement of $u$ in $\bar{w_1}w_2$. 
Denote $w'_1$ and (resp. $w'_2$) the sub-web of $w_1$ (resp. of $w_2$) such that $\bar{w'_1}w'_2= u$ and let $w''_1$ (resp. $w''_2$) be the complementary web of $w'_1$ in $w_1$. (resp. of $w'_2$ in $w_2$). Denote $\epsilon'$ the boundary of $w'_1$ and $\epsilon''$ the boundary of $w_1''$. Let $\w'$ be $(w'_1,w'_2)$ and $\w''$ be $(w''_1,w''_2)$, it's clear that $\w$ and $\w''$ belong to $W$. We have $l(\epsilon') + l(\epsilon) = l(\epsilon)$, so that $\w'$ and $\w''$ are smaller than $\w$. It's clear that $\w$ is symmetric if and only if $\w'$ and $\w''$ are. We have
\[\kup{\bar{w_1}w_2}=\kup{\bar{w'_1}w'_2}\cdot\kup{\bar{w''_1}w''_2},
\] so that if $\w'$ and $\w''$ are nice, then $\w$ is nice.
\end{proof}

\begin{lem}\label{lem:digon2nice1}
  Let  $\w =(w_1,w_2)$ be an element of $W$ such that $\bar{w_1}w_2$ is connected and contains a digon $B$ which intersects the border of $\w$ in exactly one point per side of the digon. Then there exists $\w' =(w'_1,w'_2)$ with $\w'<\w$ such that $\w'$ nice implies $\w$ nice.
\end{lem}
  \begin{figure}[h!]
    \centering
    \begin{tikzpicture}[scale= 0.8]
    \input{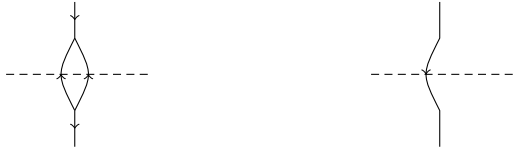}
    \end{tikzpicture}
    \caption{Case of the digon with one vertex by side~: on the left $\bar{w_1}w_2$, on the right $\bar{w'_1}w'_2$.}
    \label{fig:big-diff-side}
  \end{figure}
\begin{proof}
The situation is illustrated on figure \ref{fig:big-diff-side}.  We perform the digon reduction by deleting one edge of the digon, reversing the orientation of the other one and forgetting the two 2-valent vertices. We obtain a new pair of $\epsilon'$-webs $\w'=(w_1',w_2')$. This is an element of $W$, in fact this is clear that the two webs are superficial ; to see that they are non-elliptic, one can notice that from $w_i$ to $w'_i$ we removed just one vertex and this vertex is not adjacent to any face but the unbounded one, and consequently the non-ellipticity is preserved. It's clear that $\w$ is symmetric if and only if $\w'$ is, the length of $\epsilon'$ is equal to $l(\epsilon)-1$ so $\w'<\w$ and we have $\kup{\bar{w_1}w_2}=[2]\kup{\bar{w'_1}w'_2}$, so if $\w'$ is nice then $\w$ is nice.
\end{proof}
\begin{lem}\label{lem:digon2nice2}
  Let  $\w =(w_1,w_2)$ be an element of $W$ such that $\bar{w_1}w_2$ contains a digon $B$ which intersects the border of $\w$ in exactly two points which are on the same side. Then there exists a finite collection $\left(\w^{(i)}\right)$ with $\w^{(i)}<\w$ for all $i$, such that $\w^{(i)}$ nice for all $i$ implies $\w$ nice.
\end{lem}

\begin{figure}[h!]
    \centering
    \begin{tikzpicture}[scale= 0.8]
     \begin{scope}[scale=0.6,decoration={markings, mark=at
         position 0.5 with {\arrow{>}}},postaction={decorate}]
          \draw[densely dashed] (0.7,2.5) .. controls (0.5,2.5) and (0,2.5) .. (0,2) 
          .. controls (0,1.5) and (0.5,1.5) .. (0.7,1.5) ;
         \draw[postaction= {decorate}] (0,4) -- (0,3);
         \draw[postaction= {decorate}] (0,1) -- (0,0);
         \draw[postaction= {decorate}] (0,1) .. controls (0.5,2) and (0.5,2) ..(0,3);
         \draw[postaction= {decorate}] (0,1) .. controls (-0.5,2) and (-0.5,2) ..(0,3);    \end{scope}
    \begin{scope}[xshift = 6cm, scale=0.6,decoration={markings, mark=at
        position 0.5 with {\arrow{>}}},postaction={decorate}]
       \draw[densely dashed] (0.7,2.5) .. controls (0.5,2.5) and (0,2.5) .. (0,2) 
          .. controls (0,1.5) and (0.5,1.5) .. (0.7,1.5) ;
         \draw[postaction=decorate] (0,4) .. controls (0,3.5) and (0,3.5) ..(0,3).. controls (-0.5,2) and (-0.5,2) .. (0,1) .. controls (0,0.5) and (0,0.5).. (0,0);
    \end{scope}
    \end{tikzpicture}
    \caption{Case of the digon with the two vertices on the same side~: on the left $\bar{w_1}w_2$, on the right $\bar{w'_1}{w'_2}$.}
    \label{fig:big-same-side}
  \end{figure}
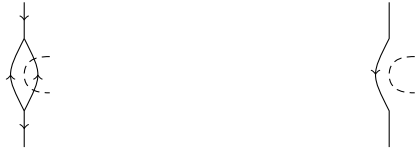

\begin{proof}
  The situation is illustrated by figure \ref{fig:big-same-side}. In this case $\w$ cannot be symmetric. Without loss of generality\footnote{The two vertices are on  the same side since one of the edges joining them does not meet the border.} we can suppose that the two vertices of $B$ are in $w_1$.  We reduce the digon $B$ as follows : delete the side meeting the border, reverse the orientation of the other edge, and forget the two 2-valent vertices denote $\w'$ the new pair of $\epsilon'$-webs corresponding to the situation. The length of $\epsilon'$ is equal to $l(\epsilon)-2$. The $\epsilon'$-web $w'_1$ is clearly superficial and semi-non-elliptic so we can apply lemma \ref{lem:sneni} and we have a finite collection $w_1^{(i)}$ of superfical non-elliptic $\epsilon'$-webs such that in the skein module  $S^{\epsilon'}$ we have $w'_1 = \sum_i \lambda w_i^{(i)}$ for some positive integer $\lambda_i$. On the other hand $w'_2$ is clearly superfical and non-elliptic. Denote $\w^{(i)}=(w_i^{(i)}, w'_2)$. Suppose that all the $\w^{(i)}$ are nice. We have :
\[
\deg\kup{\bar{w_1}w_2} = \deg\left([2]\kup{\bar{w'_1}w'_2}\right) = 1+ \max_{i} \deg \kup{\bar{w_1^{(i)}}w'_2} \leq 1 + l(\epsilon') = l(\epsilon) -1
\]
And this shows that $\w$ is nice.
\end{proof}

\begin{prop}\label{prop:digon2nice}
  Let  $\w =(w_1,w_2)$ be an element of $W$ such that $\bar{w_1}w_2$ is connected and contains a digon $B$. Then there exists a finite collection $\left(\w^{(i)}\right)$ with $\w^{(i)}<\w$ for all $i$, such that $\w^{(i)}$ nice for all $i$ implies $\w$ nice.
\end{prop}
\begin{proof}
  If we are in the situation of \ref{lem:digon2nice1} and \ref{lem:digon2nice2} this is already done, so suppose that the border meets the digon $B$ at least 3 times. The situation is illustrated on figure \ref{fig:digon2nice-moves}. In this case $\w$ cannot be symmetric. We denote $D$ the disk delimited by the digon. As the web $\bar{w_1}w_2$ is connected we can suppose that the interior of $D$ is disjoint from the web. Consider now the restriction of the border to $D$. It's the reunion of different arcs. Push one outer arc $a$ outside $D$. This leads to a new pair of $\epsilon'$-webs $\w'=(w'_1,w'_2)$. The length of $\epsilon'$ is equal to $l(\epsilon)-1$ (when the two extremities of $a$ lie on two different sides of $B$) or to $l(\epsilon)-2$ (when the two extremities of $B$ lie on the same side of $B$). The web $\w'$ is in $W$ : The operation that we described does not disturb the non-ellipticity condition neither the superficiality condition. Furthermore $\w'<\w$, and if $\w'$ is nice $\deg \kup{\bar{w_1}w_2} \leq l(\epsilon) -1$ so that $\w$ is nice.
\end{proof}
\begin{figure}[h!]
    \centering
    \begin{tikzpicture}[scale= 0.8]
    \begin{scope}[scale=0.6]
\draw (2.5,2) node {$\rightarrow$};
\draw (16.5,2) node {$\rightarrow$};
 \begin{scope}[ scale=1,decoration={markings, mark=at
        position 0.5 with {\arrow{>}}},postaction={decorate}]
         \draw[densely dashed] (1.5,2.5) -- (-1.5,2.5);
        \draw[postaction= {decorate}] (0,4) -- (0,3);
        \draw[postaction= {decorate}] (0,1) -- (0,0);
        \draw[postaction= {decorate}] (0,1) .. controls (0.5,2) and (0.5,2) ..(0,3);
        \draw[postaction= {decorate}] (0,1) .. controls (-0.5,2) and (-0.5,2) ..(0,3); 
    \end{scope}
    \begin{scope}[ xshift = 5cm, scale=1,decoration={markings, mark=at
        position 0.5 with {\arrow{>}}},postaction={decorate}]
         \draw[densely dashed] (1.5,2.5) .. controls (0.7,2.5) and (0.2, 3.2) .. (0,3.2)..  controls (-0.2,3.2) and (-0.7, 2.5) ..  (-1.5,2.5);
        \draw[postaction= {decorate}] (0,4) -- (0,3);
        \draw[postaction= {decorate}] (0,1) -- (0,0);
        \draw[postaction= {decorate}] (0,1) .. controls (0.5,2) and (0.5,2) ..(0,3);
        \draw[postaction= {decorate}] (0,1) .. controls (-0.5,2) and (-0.5,2) ..(0,3); 
    \end{scope}
 \begin{scope}[xshift= 14cm, scale=1,decoration={markings, mark=at
         position 0.5 with {\arrow{>}}},postaction={decorate}]
          \draw[densely dashed] (0.9,2.5) .. controls (0.5,2.5) and (0,2.5) .. (0,2) 
          .. controls (0,1.5) and (0.5,1.5) .. (0.9,1.5) ;
         \draw[postaction= {decorate}] (0,4) -- (0,3);
         \draw[postaction= {decorate}] (0,1) -- (0,0);
         \draw[postaction= {decorate}] (0,1) .. controls (0.5,2) and (0.5,2) ..(0,3);
         \draw[postaction= {decorate}] (0,1) .. controls (-0.5,2) and (-0.5,2) ..(0,3);    
       \end{scope}
    \begin{scope}[xshift = 19cm, scale=1,decoration={markings, mark=at
        position 0.5 with {\arrow{>}}},postaction={decorate}]
         \draw[densely dashed] (0.9,2.5) .. controls (0.7,2.5) and (0.6,2.5) .. (0.6,2) 
          .. controls (0.6,1.5) and (0.7,1.5) .. (0.9,1.5) ;
         \draw[postaction= {decorate}] (0,4) -- (0,3);
         \draw[postaction= {decorate}] (0,1) -- (0,0);
         \draw[postaction= {decorate}] (0,1) .. controls (0.5,2) and (0.5,2) ..(0,3);
         \draw[postaction= {decorate}] (0,1) .. controls (-0.5,2) and (-0.5,2) ..(0,3);     
    \end{scope}

\end{scope}[scale =0.6]
    \end{tikzpicture}
    \caption{Remaining cases for the digon : we move the boundary. On the left the outer arc meets the two edges, on the right it meets two times the same edge. In both cases only the outer arc of the boundary is drawn, but the boundary meets the digon elsewhere as well.}
    \label{fig:digon2nice-moves}
  \end{figure}
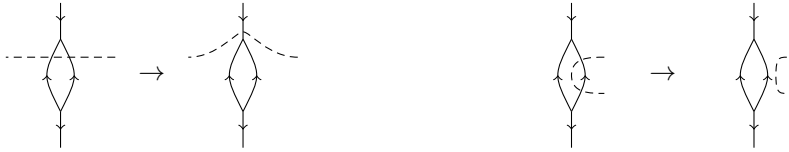
  \begin{lem}\label{lem:square:Opp-side}
Let  $\w =(w_1,w_2)$ be an element of $W$ such that $\bar{w_1}w_2$ is connected and contains a square $S$ such that $S$ intersects the border of $\w$ in two points on opposite sides. Then there exists a finite collection $\left(\w^{(k)}\right)$ with $\w^{(k)}<\w$ for all $k$, such that if all the $\w^{(k)}$ are nice then $\w$ is nice. 
  \end{lem}
\begin{figure}[h!]
    \centering
    \begin{tikzpicture}[scale = 0.8]
    \begin{scope}[scale={1},decoration={markings, mark=at
     position 0.5 with {\arrow{>}}},postaction={decorate}]
%       \draw[dotted] (1.5,1.5) circle (1.414 cm);
   \draw[densely dashed] (0,1.5) --(3,1.5);
      \draw[postaction = {decorate}] (0.5,0.5)--(1,1);
      \draw[postaction = {decorate}] (2.5,2.5)--(2,2);
      \draw[postaction = {decorate}] (2,1)--(1,1);
      \draw[postaction = {decorate}] (2,1)--(2,2);
      \draw[postaction = {decorate}] (2,1)--(2.5,0.5);
      \draw[postaction = {decorate}] (1,2)--(2,2);
      \draw[postaction = {decorate}] (1,2)--(1,1);
      \draw[postaction = {decorate}] (1,2)--(0.5,2.5);
 \end{scope}

\begin{scope}[xshift = 4cm, scale={1},decoration={markings, mark=at
     position 0.5 with {\arrow{>}}},postaction={decorate}]
      \draw[densely dashed] (0,1.5) --(3,1.5);
 %      \draw[dotted] (1.5,1.5) circle (1.414 cm);
      \draw[postaction = {decorate}] (0.5,0.5) .. controls (1,1.5) and (1,1.5) .. (0.5, 2.5);
      \draw[postaction = {decorate}] (2.5,2.5).. controls (2,1.5) and (2, 1.5) ..(2.5,0.5);  
 \end{scope}

\begin{scope}[xshift = 8cm, scale={1},decoration={markings, mark=at
     position 0.5 with {\arrow{>}}},postaction={decorate}]
  %  \draw[dotted] (1.5,1.5) circle (1.414 cm);
      \draw[densely dashed] (0,1.5) --(3,1.5);
      \draw[postaction = {decorate}] (0.5,0.5) .. controls (1.5,1) and (1.5,1) .. (2.5, 0.5);
      \draw[postaction = {decorate}] (2.5,2.5).. controls (1.5,2) and (1.5,2) ..(0.5,2.5);  
 \end{scope}
    \end{tikzpicture}
    \caption{The border meets two opposite sides of the square : from left to right~: $\bar{w_1}w_2$, $\bar{w'_1}w'_2$ and $\bar{w''_1}{w''_2}$. It's clear that $w'_1$ and $w'_2$ are superficial non-elliptic, and that $w''_1$ and $w''_2$ are superfical semi-non-ellipitic. Furthermore, $w_1=w_2$ if and only if $w'_1=w'_2$.}
    \label{fig:square-sym}
  \end{figure}
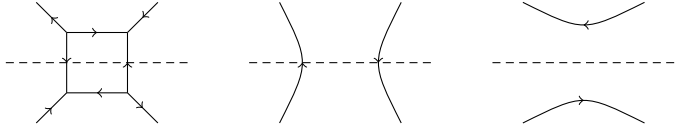
  \begin{proof} The situation is illustrated by figure \ref{fig:square-sym}.
The connectedness hypothesis tells us that we can suppose the interior of the disk $D$ delineated by $S$ to be disjoint from the web. In $D$, the border is just a simple arc joining two opposite sides. We perform the two reductions of the square by deleting two opposite sides, reversing the orientations on the two lasting sides and forgetting the four 2-valent vertices. We obtain one pair of $\epsilon$-webs  $\w'=(w'_1,w_2')$ (when keeping the sides which meet the border) and one pair of $\epsilon''$-webs $\w''=(w''_1,w''_2)$ with $l(\epsilon'')=l(\epsilon)-2$ (when deleting the sides which meet the border). The $\epsilon$-webs $w'_1$ and $w'_2$ are superficial and non-elliptic so that $\w'\in W$. Because of the number of vertices we have $\w' < \w$. The $\epsilon''$-webs $w''_1$ and $w''_2$ are superficial and semi-non-elliptic. Then there exists a finite collection $\left(w_1^{(i)}\right)$  (resp.~{}$\left(w_2^{(j)}\right)$) of superficial non-elliptic $\epsilon''$-webs and some positive integers $\lambda_i$ (resp.~{}$\mu_j$) such that in the skein module $S^{\epsilon''}$ we have $
w''_1 = \sum_i \lambda_i w_1^{(i)}$ and $ w''_2 = \sum_j \mu_j w_2^{(j)}$. We have then :
\[
\kup{\bar{w_1}w_2} = \kup{\bar{w'_1}w'_2} + \sum_{i,j}\lambda_i\mu_j\kup{\bar{w_1^{(i)}}w_2^{(j)}}.
\] We denote $\w^{(i,j)}=(w_1^{(i)},w_2^{(j)})$. Suppose that $\w'$ is nice and that all the $\w^{(i,j)}$ are nice. It's straightforward that $\w$ is symmetric if and only if $\w'$ is, so that it's clear that, $\w$ is nice.
  \end{proof}
  \begin{lem}\label{lem:square:adj-side}
Let  $\w =(w_1,w_2)$ be an element of $W$ such that $\bar{w_1}w_2$ is connected and contains a square $S$ such that $S$ intersects the border of $\w$ in two points on adjacent sides. Then there exists a finite collection $\left(\w^{(k)}\right)$ with $\w^{(k)}<\w$ for all $k$, such that if all the $\w^{(k)}$ are nice then $\w$ is nice.     
  \end{lem}
 \begin{figure}[h!]
   \centering
   \begin{tikzpicture}[scale= 0.8]
     \begin{scope}[scale={1},decoration={markings, mark=at
     position 0.5 with {\arrow{>}}},postaction={decorate}]
       %\draw[dotted] (1.5,1.5) circle (1.414 cm);
       \draw[densely dashed] (0.5,1.5) .. controls (1,1.5) and (1.5,1) .. (1.5, 0.5);
      \draw[postaction = {decorate}] (0.5,0.5)--(1,1);
      \draw[postaction = {decorate}] (2.5,2.5)--(2,2);
      \draw[postaction = {decorate}] (2,1)--(1,1);
      \draw[postaction = {decorate}] (2,1)--(2,2);
      \draw[postaction = {decorate}] (2,1)--(2.5,0.5);
      \draw[postaction = {decorate}] (1,2)--(2,2);
      \draw[postaction = {decorate}] (1,2)--(1,1);
      \draw[postaction = {decorate}] (1,2)--(0.5,2.5);
 \end{scope}
\begin{scope}[xshift = 5cm, scale={1},decoration={markings, mark=at
     position 0.5 with {\arrow{>}}},postaction={decorate}]
       %\draw[dotted] (1.5,1.5) circle (1.414 cm);
       \draw[densely dashed] (0.5,1.5) .. controls (0.8,1.5) and (0.7,1.1) .. (0.9, 0.9) .. controls (1.1,0.7) and (1.5,0.8).. (1.5,0.5);
      \draw[postaction = {decorate}] (0.5,0.5)--(1,1);
      \draw[postaction = {decorate}] (2.5,2.5)--(2,2);
      \draw[postaction = {decorate}] (2,1)--(1,1);
      \draw[postaction = {decorate}] (2,1)--(2,2);
      \draw[postaction = {decorate}] (2,1)--(2.5,0.5);
      \draw[postaction = {decorate}] (1,2)--(2,2);
      \draw[postaction = {decorate}] (1,2)--(1,1);
      \draw[postaction = {decorate}] (1,2)--(0.5,2.5);
 \end{scope}
   \end{tikzpicture}
   \caption{The boundary meets the square in two adjacent sides~: on the left $\bar{w_1}w_2$, on the right $\bar{w'_1}w'_2$. It's clear that $w'_2$ is superficial non-elliptic, and that $w'_2$ is superficial semi-non-ellipitic. }
   \label{fig:square-ad-sides}
 \end{figure}
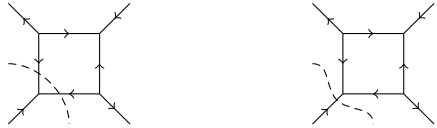
  \begin{proof} The situation is illustrated by figure \ref{fig:square-ad-sides}.
    First notice that in this situation $\w$ cannot be symmetric. As before, the connectedness hypothesis allows us to suppose that the interior of the disk $D$ delineated by $S$ is disjoint from the web. The restriction of the border to $D$ is just a simple arc connecting two adjacent sides. We move a little this arc : we push it outside the square through the common vertex of the two adjacent sides and we obtain a new pair of $\epsilon'$-webs $\w'=(w'_1,w'_2)$ with $l(\epsilon')=l(\epsilon)-1$. Now the square $S$ doesn't meet the border anymore. With no loss of generality we can assume that it lies in $w'_1$. The $\epsilon'$-web $w'_2$ is superficial and non-elliptic. The $\epsilon'$-web $w'_1$ is superficial (the only faces which could be nested in $w'_1$ are the one next to the square, but they are obviously not) and semi-non-elliptic. So it exists a finite collection $\left(w_1^{(k)}\right)$ of non-elliptic superficial $\epsilon'$-webs and some positive integers $\lambda_k$ such that in the skein module $S^{\epsilon'}$, $w_1 =\sum_k \lambda_k w_1^{(k)}$. Let $\w^{(k)} = (w_1^{(k)},w'_2)$. It's clear that $\w^{(k)}<\w$ for all $k$. Suppose  that all the $\w^{(k)}$ are nice. \[\deg \kup{\bar{w_1}w_2} = \deg\left( \sum_k \kup{\bar{w_1^{(k)}}w_2}\right) \leq l(\epsilon') = l(\epsilon) -1.\]
This proves that $\w$ is nice.
  \end{proof}
We will now inspect the case where the border meets just one side of the square. This is the most technical part, so we need to separate it in different sub-cases. For this we will need to consider the face adjacent to the square and to the side opposite to the one that meets the border. We call this face the \emph{opposed face}.

\begin{lem}\label{lem:square_same_side0}
Let  $\w =(w_1,w_2)$ be an element of $W$ such that $\bar{w_1}w_2$ is connected and contains a square $S$. Suppose $S$ intersects the border of $\w$ in exactly two points on the same side and the opposed face $F$ or one of its neighbors different from the square meets the boundary. Then there exists a finite collection $\left(\w^{(k)}\right)$ with $\w^{(k)}<\w$ for all $k$, such that if $\w^{(k)}$ is nice for all $k$ then  $\w$ is nice.     
 \end{lem}
 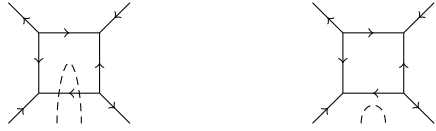
\begin{figure}[h!]
   \centering
   \begin{tikzpicture}[scale= 0.8]
     \begin{scope}[scale={1},decoration={markings, mark=at
     position 0.5 with {\arrow{>}}},postaction={decorate}]
       %\draw[dotted] (1.5,1.5) circle (1.414 cm);
       \draw[densely dashed] (1.3,0.5) .. controls (1.3,1) and (1.4,1.5) .. (1.5, 1.5).. controls (1.6,1.5) and (1.7,1) .. (1.7,0.5);
      \draw[postaction = {decorate}] (0.5,0.5)--(1,1);
      \draw[postaction = {decorate}] (2.5,2.5)--(2,2);
      \draw[postaction = {decorate}] (2,1)--(1,1);
      \draw[postaction = {decorate}] (2,1)--(2,2);
      \draw[postaction = {decorate}] (2,1)--(2.5,0.5);
      \draw[postaction = {decorate}] (1,2)--(2,2);
      \draw[postaction = {decorate}] (1,2)--(1,1);
      \draw[postaction = {decorate}] (1,2)--(0.5,2.5);
 \end{scope}
\begin{scope}[xshift = 5cm, scale={1},decoration={markings, mark=at
     position 0.5 with {\arrow{>}}},postaction={decorate}]
       %\draw[dotted] (1.5,1.5) circle (1.414 cm);
       \draw[densely dashed] (1.3,0.5) .. controls (1.3,0.7) and (1.4,0.8) .. (1.5, 0.8).. controls (1.6,0.8) and (1.7,0.7) .. (1.7,0.5);
      \draw[postaction = {decorate}] (0.5,0.5)--(1,1);
      \draw[postaction = {decorate}] (2.5,2.5)--(2,2);
      \draw[postaction = {decorate}] (2,1)--(1,1);
      \draw[postaction = {decorate}] (2,1)--(2,2);
      \draw[postaction = {decorate}] (2,1)--(2.5,0.5);
      \draw[postaction = {decorate}] (1,2)--(2,2);
      \draw[postaction = {decorate}] (1,2)--(1,1);
      \draw[postaction = {decorate}] (1,2)--(0.5,2.5);
 \end{scope}
   \end{tikzpicture}
   \caption{Illustration of lemma \ref{lem:square_same_side0} : on the left $\bar{w_1}w_2$, on the right $\bar{w'_1}w'_2$. It's clear that $w'_2$ is superficial non-elliptic, and that $w'_2$ is superfical semi-non-ellipitic. }
   \label{fig:square-same-side0}
 \end{figure}
 \begin{proof}
First notice that in this case $\w$ cannot be symmetric. The four vertices of $S$ lie on the same side of the border, we can suppose that this is on the $w_1$ side.
As usual, we can suppose that the interior of the disk $D$ delineated by the square is disjoint from the web. The restriction of the border is just a simple arc in $D$ joining one side to itself. We move the border locally by pushing it away from $S$ in $w_2$. We obtain a new pair of $\epsilon'$-webs $\w'=(w'_1,w'_2)$, with $l(\epsilon') = l(\epsilon) -2$. It's clear that $w'_1$ is superficial and  semi-non-elliptic. So it exists a finite collection $\left(w_1^{(k)}\right)$ of non-elliptic superficial $\epsilon'$-webs and some positive integers $\lambda_k$ such that in the skein module  $S^{\epsilon'}$, $w_1 =\sum_k \lambda_k w_1^{(k)}$. Set $\w^{(k)} = (w_1^{(k)},w'_2)$.  Suppose  that all the $\w^{(k)}$ are nice. Then $\deg \kup{\bar{w_1}w_2}\leq l(\epsilon')-2 < l(\epsilon)$ and hence $\w$ is nice.
 \end{proof}

 \begin{lem}\label{lem:square_same_side1}
   Let  $\w =(w_1,w_2)$ be an element of $W$ such that $\bar{w_1}w_2$ is connected, contains no digon and contains a square $S$. Suppose $S$ intersects the border of $\w$ in exactly two points on the same side and the opposed $F$ face has at least 8 sides. Then there exists a finite collection $\left(\w^{(k)}\right)$ with $\w^{(k)}<\w$ for all $k$, such that $\w^{(k)}$ nice for all $k$ implies $\w$ nice.  
 \end{lem}
 \begin{figure}[h!]
   \centering
   \begin{tikzpicture}[scale= 0.8]
     \begin{scope}[scale={1},decoration={markings, mark=at
     position 0.5 with {\arrow{>}}},postaction={decorate}]
       %\draw[dotted] (1.5,1.5) circle (1.414 cm);
       \draw[densely dashed] (1.3,0.5) .. controls (1.3,1) and (1.4,1.5) .. (1.5, 1.5).. controls (1.6,1.5) and (1.7,1) .. (1.7,0.5);
      \draw[postaction = {decorate}] (0.5,0.5)--(1,1);
      \draw[postaction = {decorate}] (2.5,2.5)--(2,2);
      \draw[postaction = {decorate}] (2,1)--(1,1);
      \draw[postaction = {decorate}] (2,1)--(2,2);
      \draw[postaction = {decorate}] (2,1)--(2.5,0.5);
      \draw[postaction = {decorate}] (1,2)--(2,2);
      \draw[postaction = {decorate}] (1,2)--(1,1);
      \draw[postaction = {decorate}] (1,2)--(0.5,2.5);
 \end{scope}
\begin{scope}[xshift = 4cm, scale={1},decoration={markings, mark=at
     position 0.5 with {\arrow{>}}},postaction={decorate}]
  \draw[densely dashed] (1.3,0.5) .. controls (1.3,1) and (1.4,1.5) .. (1.5, 1.5).. controls (1.6,1.5) and (1.7,1) .. (1.7,0.5);
      \draw[postaction = {decorate}] (0.5,0.5) .. controls (1,1.5) and (1,1.5) .. (0.5, 2.5);
      \draw[postaction = {decorate}] (2.5,2.5).. controls (2,1.5) and (2, 1.5) ..(2.5,0.5);  
 \end{scope}
\begin{scope}[xshift = 8cm, scale={1},decoration={markings, mark=at
     position 0.5 with {\arrow{>}}},postaction={decorate}]
\draw[densely dashed] (1.3,0.5) .. controls (1.3,1) and (1.4,1.5) .. (1.5, 1.5).. controls (1.6,1.5) and (1.7,1) .. (1.7,0.5);
     \draw[postaction = {decorate}] (0.5,0.5) .. controls (1.5,1) and (1.5,1) .. (2.5, 0.5);
      \draw[postaction = {decorate}] (2.5,2.5).. controls (1.5,2) and (1.5,2) ..(0.5,2.5);  
 \end{scope}
   \end{tikzpicture}
   \caption{Illustration of lemma \ref{lem:square_same_side1} : from left to right : $\bar{w_1}w_2$, $\bar{w'_1}w'_2$ and $\bar{w''_1}w''_2$.}
   \label{fig:square-same-side1}
 \end{figure}
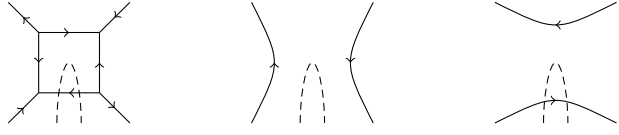
 \begin{proof}
The situation is illustrated on figure \ref{fig:square-same-side1}
First notice that in this case $\w$ cannot be symmetric. The four vertices of $S$ lie on the same side of the border, we can suppose that this is on the $w_1$ side.
As usual, we can suppose that the interior of the disk $D$ delineated by the square is disjoint from the web. The restriction of the border is just a simple arc in $D$ joining one side to itself. We perform the two reductions of the square by deleting two opposite sides, reversing the orientations on the two lasting sides and forgetting the four 2-valent vertices. We obtain one pair of $\epsilon'$-webs  $\w'=(w'_1,w_2')$ with $l(\epsilon')=l(\epsilon)-2$ (when deleting the sides which meet the border) and one pair of $\epsilon$-webs $\w''=(w''_1,w''_2)$  (when keeping the sides which meet the border). The $\epsilon'$-webs $w'_1$ and $w'_2$ are superficial and $1$-elliptic. Thanks to lemma \ref{lem-1-elliptic} there exist a finite collection $\left(w_1^{(i)}\right)$ of non-elliptic and superficial $\epsilon'$-webs and a finite collection $(P_i)$ of symmetric Laurent polynomial in $\NN[q,q^{-1}]$ with degree at most 1 such that in the skein module $S^{\epsilon'}$, $w'_1 =\sum_i P_i w_1^{(i)}$. On the other hand, $w'_2$ is superficial and non-elliptic. Denote $\w^{(i)}=(w_1^{(i)},w'_2)$. Let us inspect $\w''$ now, it's clear that the $\epsilon$-web $w''_1$  is superficial and the hypothesis made on $F$ implies that it is non-elliptic. The $\epsilon$-web $w''_2$  is clearly non-elliptic and superficial. The hypothesis on the absence of digon implies that $\w''$ is not symmetric. Suppose that $\w''$ and all the $\w^{(k)}$ are nice. Then we have :
\begin{align*}\deg\kup{\bar{w_1}w_2} &= \deg\left(\kup{\bar{w'_1}w'_2} + \kup{\bar{w''_1}w''_2} \right)=  \deg\left(\sum P_i\kup{\bar{w_1^{(i)}}w'_2} + \kup{\bar{w''_1}w''_2} \right) \\ &\leq \max (1 + l(\epsilon'), l(\epsilon) -1) = l(\epsilon) -1.
\end{align*}
This proves that $\w$ is nice.
 \end{proof}
 \begin{lem}\label{lem:square_same_side2}
   Let  $\w =(w_1,w_2)$ be an element of $W$ such that $\bar{w_1}w_2$ is connected, contains no digon and contains a square $S$. Suppose $S$ intersects the border of $\w$ in exactly two points on the same side and the opposed face $F$ is an hexagon and does not meet the border. Then there exists a finite collection $\left(\w^{(k)}\right)$ with $\w^{(k)}<\w$ for all $k$, such that $\w^{(k)}$ nice for all $k$ implies $\w$ nice.  
 \end{lem}
 \begin{proof}
First notice that in this case $\w$ cannot be symmetric. The four vertices of $S$ lie on the same side of the border, we can suppose that this is on the $w_1$ side. As usual, we can suppose that the interior of the disk $D_1$ delineated by the $S$ and interior of the disk $D_2$ delineate by the hexagon are disjoint from the web. The restriction of the border is just a simple arc in $D_1$ joining one side to itself. We move the border by pushing the arc out of $D_1$ (this is the same move as in figure \ref{fig:square-same-side0}). We denote $\w'=(w'_1,w'_2)$ the new pair of $\epsilon'$-webs with $l(\epsilon') = l(\epsilon)-2$. The $\epsilon'$-web $w'_1$ is clearly semi-superficial and the $\epsilon'$-web $w_2'$ is superficial and non-elliptic. The lemma \ref{lem-semi-superficial} tells us that there exists a finite collection $\left(w_1^{(i)}\right)$ of non-elliptic superficial $\epsilon'$-webs and a finite collection $(P_i)$ of symmetric Laurent polynomial in $\NN[q,q^{-1}]$ with degree at most 1 such that in the skein module $S^{\epsilon'}$, $w'_1 =\sum_i P_i w_1^{(i)}$. Denote $\w^{(i)}=(w_1^{(i)},w'_2)$, and suppose that all the $\w^{(i)}$ are nice. Then 
\[
\deg\kup{\bar{w_1}w_2} = \deg\left(\sum P_i\kup{\bar{w_1^{(i)}}w'_2} \right) \leq 1+l(\epsilon') =l(\epsilon)-1
\]
This proves that $\w$ is nice.
\end{proof}
\begin{prop}\label{prop:square2nice}
 Let  $\w =(w_1,w_2)$ be an element of $W$ such that $\bar{w_1}w_2$ is connected, contains no digon and contains a square $S$. Then there exists a finite collection $\left(\w^{(k)}\right)$ with $\w^{(k)}<\w$ for all $k$, such that $\w^{(k)}$ nice for all $k$ implies $\w$ nice. 
\end{prop}
\begin{proof}
The border must cut the square $S$. prove do this by induction on the number of intersection points of the border with $S$. If the border meets $S$ two times then the lemmas \ref{lem:square:Opp-side}, \ref{lem:square:adj-side}, \ref{lem:square_same_side0}, \ref{lem:square_same_side1} and \ref{lem:square_same_side2} give the result. If it has more that two intersections points, then $\w$ is not symmetric. We move apart of the border (an outer arc) outside $S$ without increasing the length of $\epsilon$ (in case the arc meets adjacent sides we do like in \ref{lem:square:adj-side}, in case it meets the same side two times we do the move described on figure \ref{fig:square-same-side1}, in case all the outer arcs meet two opposite sides we perform the move depicted on figure~\ref{fig:square-remaing-move}). The only thing to realize is that when one move an arc joining to opposite side of $S$ out of $S$ this cannot result to a symmetric $\w$ because of the no-digon hypothesis. These moves decrease the number of intersecting points of $S$ with the border and we can use the recursion hypothesis. 
 \end{proof}
\begin{figure}[h!]
  \centering
  \begin{tikzpicture}[scale = 0.8]
     \begin{scope}[scale={1},decoration={markings, mark=at
     position 0.5 with {\arrow{>}}},postaction={decorate}]
%       \draw[dotted] (1.5,1.5) circle (1.414 cm);
   \draw[densely dashed] (0.086,1.5) --(2.914,1.5);
      \draw[postaction = {decorate}] (0.5,0.5)--(1,1);
      \draw[postaction = {decorate}] (2.5,2.5)--(2,2);
      \draw[postaction = {decorate}] (2,1)--(1,1);
      \draw[postaction = {decorate}] (2,1)--(2,2);
      \draw[postaction = {decorate}] (2,1)--(2.5,0.5);
      \draw[postaction = {decorate}] (1,2)--(2,2);
      \draw[postaction = {decorate}] (1,2)--(1,1);
      \draw[postaction = {decorate}] (1,2)--(0.5,2.5);
 \end{scope}

\begin{scope}[xshift= 5cm, scale={1},decoration={markings, mark=at
     position 0.5 with {\arrow{>}}},postaction={decorate}]
 %      \draw[dotted] (1.5,1.5) circle (1.414 cm);
   \draw[densely dashed] (0.086,1.5) .. controls (0.7,1.5) and (0.4,0.7) ..(1.5,0.7).. controls (2.6,0.7) and (2.3,1.5)..(2.914,1.5);
      \draw[postaction = {decorate}] (0.5,0.5)--(1,1);
      \draw[postaction = {decorate}] (2.5,2.5)--(2,2);
      \draw[postaction = {decorate}] (2,1)--(1,1);
      \draw[postaction = {decorate}] (2,1)--(2,2);
      \draw[postaction = {decorate}] (2,1)--(2.5,0.5);
      \draw[postaction = {decorate}] (1,2)--(2,2);
      \draw[postaction = {decorate}] (1,2)--(1,1);
      \draw[postaction = {decorate}] (1,2)--(0.5,2.5);
 \end{scope}
  \end{tikzpicture}
  \caption{When all the outer arc meet the two side, we move one of this arc out of the square.}
  \label{fig:square-remaing-move}
\end{figure}
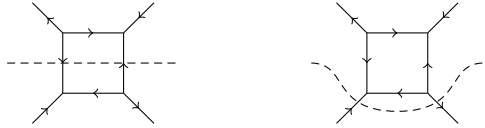
 \begin{proof}[Proof of theorem \ref{thm:superficial2monic}]
   As we said, this is enough to prove \ref{prop:superfical2monic}. We do this by induction with respect to the order on $W$. It's clear that the result is true for $(\emptyset,\emptyset)$. Suppose we have an element $\w = (w_1,w_2)$ of $W$ such that for all $\w'$ of $W$ with $\w'<\w$ then $\w'$ is nice. Then depending on how $\bar{w_1}w_2$ looks like we can apply lemma~\ref{lem:notconnected2nice}, lemma \ref{lem:notconnected2nice}, proposition~\ref{prop:digon2nice} or proposition~\ref{prop:square2nice}, and this shows that $\w$ is nice.
 \end{proof}
\bibliographystyle{alpha}
\bibliography{../biblio}
\end{document}